\documentclass{article}

\usepackage{amssymb}
\usepackage{amsfonts}
\usepackage{amsmath}
\usepackage{tikz}
\usepackage{amsthm}
\usepackage{graphicx,subfigure}
\usepackage[all]{xy}
\usepackage{tikz-cd}
\usepackage{mathrsfs}
 \usepackage{comment} 
\usepackage{tabularx}
\usepackage{url}
\usepackage{hyperref}
\newtheorem{eg}{Example}[section] 
\newtheorem{theo}{Theorem}[section]
\newtheorem{cor}{Corollary}[section]
\newtheorem{lemma}{Lemma}[section]
\newtheorem{prop}{Proposition}[section]
\newtheorem{defi}{Definition}[section]
\newtheorem{Notation}{Notation}[section]
\newtheorem{remark}{Remark}[section]
\usepackage[section]{placeins} 
\usepackage{float} 
\def\n{\mathbb{N}}

\def\z{\mathbb{Z}}
\def\c{\mathbb{C}}

\begin{document}


\title{Rack Representations and Connections with groups representations}

\author{JOS\'E~GREGORIO~RODR\'IGUEZ-NIETO\\Escuela de Matem\'aticas - Facultad de Ciencias,\\Universidad Nacional de Colombia,\\Carrera 65 N. 59A-110, Medell\'in-Colombia.\\e-mail:~\url{jgrodrig@unal.edu.co} \\\mbox{}\\ OLGA~PATRICIA~SALAZAR-D\'IAZ\\Escuela de Matem\'aticas - Facultad de Ciencias,\\Universidad Nacional de Colombia,\\Carrera 65 N. 59A-110, Medell\'in-Colombia.\\e-mail:~\url{opsalazard@unal.edu.co} \\\mbox{}\\ RICARDO~ESTEBAN~VALLEJOS-CIFUENTES\\Instituto de Matem\'aticas\\Universidad de Antioquia,\\Calle 67 No. 53 - 108, Medell\'in-Colombia.\\e-mail:~\url{ricardo.vallejos@udea.edu.co} \\\mbox{}\\ RA\'UL~VEL\'ASQUEZ\\Instituto de Matem\'aticas\\Universidad de Antioquia,\\Calle 67 No. 53 - 108, Medell\'in-Colombia.\\e-mail:~\url{raul.velasquez@udea.edu.co}}

\maketitle
\begin{abstract}
In this paper we study some algebraic properties of the rack structure as well as the representation theory of it, following the ideas given by M. Elhamdadi and E. M. Moutuou in \cite{Elhamdadi}. We establish a correspondence between the irreducible strong representations of a finite and connected rack with the irreducible representation of its finite enveloping group which allows to use techniques of the latter topic in the other setting.
\end{abstract}

\textbf{keyword}: Racks, Quandles, Rack Representations, Group Representations.


\section{Introduction}
The relevancy of representation theory in the development and approach of certain problems concerning to algebraic objects is very well known, not only in pure mathematics, but also in other branches of science like physics, chemistry, dynamics, etc. In particular, such theory has been influential and determinant in the advancement of group theory and its connection with geometry through the concept of orbifolds. Because of that, we have many more results in that area than in the other ones. Besides, we have a non-associative emerging structure, called \textit{racks}, that, although it is not so new, it plays an important role in the coquecigrue problem, see \cite{MR2286884} and \cite{MR3578405}, and in the solution of the Yang-Baxter equation, see \cite{MR1722951} and \cite{MR3868941}. In $2018$ M. Elhamdadi and E. M. Moutuou in \cite{Elhamdadi} introduced the definition of representations for this object.

In the early  1980s, Joyce \cite{Joyce2} and Mateev \cite{matveev1982distributive} introduced, independently, what Joyce called a \textit{quandle}, in their aims to construct a complete knot invariant. Later, in 1993, Fenn and Rourke introduced in \cite{fFenn} the concept of \textit{rack}, as a generalization of quandle. A \textit{rack} is a set $X$ endowed with a non-associative binary operation $\vartriangleright$, such that, for all $ x\in X$, the left multiplication map $L_{x}: X\longrightarrow X$ defined by $L_{x}(y):= x \vartriangleright y$ is a bijective function, and for all $ x,y,z\in X$, we have 
\[x\vartriangleright (y\vartriangleright z) =(x\vartriangleright y)\vartriangleright (x\vartriangleright z).\]

In the case that, for all $x \in X$, we have $x \vartriangleright x = x$, then the rack $(X,\vartriangleright )$ is called a \textit{quandle}. The \textit{trivial rack} is the rack on which $x\vartriangleright y=y$, for all $x,y\in X$. The concepts of \textit{sub rack}, \textit{subquabdle}, \textit{rack homomorphism} and \textit{rack isomorphism} are given in the same fashion as in the case of groups. For every pair of elements $x,y$ of a rack $X$, we write $x\vartriangleright^{-1} y$ for the inverse function $L^{-1} _{x}(y)$. If $(X,\vartriangleright)$ satisfies that $L_{x} \circ L_{x}=id$,  for all $x \in X$, then $X$ is called an \textit{involutive quandle} or an \textit{involutive rack}, respectively. Observe that, if $X$ is an involutive quandle then we have  $ x \vartriangleright(x\vartriangleright y) = y$; \ \ for all $x,y \in X$.

Let $G$ be a group. Then, there are two natural quandles associated to $G$, they are called the \textit{conjugation} and the \textit{core} quandles. The set $G$ with the operation $g \vartriangleright h := ghg^{-1}$ is the \textit{conjugacy quandle}, and it is denoted by \textit{Conj(G)}. The set $G$ endowed  with the binary operation $g \vartriangleright h := gh^{-1}g$ is the \textit{core quandle}, and denoted by  \textit{Core(G)}, see \cite{Joyce} and \cite{Algebraofknots2} for more details. 

Due to the diverse range of applications that both racks and quandles have, it is important to study these objects in a pure way as algebraic entities in their own right, rather than solely based on their connections with other branches of mathematics, several researchers have adopted this approach, and even some recent investigations have begun with the study of the \textit{rack representation theory} \cite{Elhamdadi}. In this work, we study rack and quandle structures from a purely algebraic perspective, with a special focus on rack representation theory.
Elhamdadi and Moutuou, introduced rack representation theory in \cite{Elhamdadi} and established some general properties.  Additionally, they introduce the concept of \textit{strong representation} and stated the following theorem:   \textit{“Theorem 9.11: Every strong irreducible representation of a finite connected involutive rack is one–dimensional”}.
Unfortunately, we have found that this theorem is not true. We show a relation between strong irreducible representations of finite connected racks and irreducible representations of finite groups, through this relation we construct an infinite family of counterexamples to this theorem, see Example \ref{Example 2.10}. Specifically, we prove the next result.

    \textit{\textbf{Theorem 6.1.} Let $X$ be a finite-connected rack and let $(V,\rho)$ be a strong representation of $X$. Then $\rho$ induces a  representation $\bar{\rho}: \overline{G_{X}} \longrightarrow GL(V)$ of the finite enveloping group $\overline{G_{X}}$ such that $\bar{\rho}_{g_{x} \langle g_{x_{0}}^{n}\rangle} := \rho_{x}$ for all $x \in X$. Furthermore, if $\rho$ is an irreducible rack representation, then $\bar{\rho}$ is an irreducible group representation.}\\

Analogous to the theory of group representations, there exists a concept of equivalence between rack representations. The relation between strong irreducible representations of finite connected racks and irreducible group representations establish in Theorem \ref{Theorem 2.4} ensures that this equivalence relation is preserved through such a connection, as we show in Theorem \ref{Theorem 2.5}.\\

   \textit{\textbf{Theorem 6.2.} Let $X$ be a finite-connected rack, and let $(V,\rho)$  and $(V,\phi)$  be strong representations of $X$ such that $\rho \sim \phi$. Then, the induced group representations $\bar{\rho}: \overline{G_{X}} \longrightarrow GL(V)$  and $\bar{\phi}: \overline{G_{X}} \longrightarrow GL(V')$ are also equivalent.}\\

We also prove that the reciprocal is true for the case in which the center of $\overline{G_{X}}$ is trivial, obtaining a bijective correspondence between the equivalence class of a strong representation of finite connected racks and the equivalence class of a group representation of the finite enveloping group $\overline{G_{X}}$ with trivial center.                

\section{Groups related to a rack}\label{Section 1.2} 

Let us begin this section with the following definition.

\begin{defi}
Let $X$ be a rack. Let $F(X)$ be the free group on the set $X$ and let $N$ be the normal subgroup generated by the words of the form $(x \vartriangleright y)xy^{-1}x^{-1}$ where $x,y \in X$. We define the \textit{\textbf{associated group}}, denoted by $As(X)$, to be the quotient group $F(X)/N$, i.e. 

\begin{center}
    $As(X) = F(X) /\langle x \vartriangleright y=xyx^{-1} ,  \ \ x,y \in X \rangle $
\end{center}
 This group is also called the \textit{\textbf{enveloping group}}.
 \end{defi}
 There is an interesting functorial relation between racks and some set-theoretic solutions to the Yang-Baxter equation, see \cite{lebed_vendramin_2019} for more details. Under the context of Yang-Baxter equation the associated group $As(X)$ is the same \textit{structure group} defined in \cite{MR1722951}.\\
 Observe that,  we have two onto maps, the inclusion map $\iota: X \hookrightarrow F(X)$ defined by $\iota(x) :=x$ for all $x \in X$ and the canonical homomorphism $\pi:  F(X) \longrightarrow As(X)$  defined by $\pi(x):= \bar{x}$ for all $x \in X$. Thus, we have a natural onto map, $\eta_{X}: X \longrightarrow As(X)$ defined by $\eta_{X}(x):= (\pi \circ i)(x) = \bar{x}$. This map is not necessarily injective, as we will see in Example \ref{Example 1.12}. 

\begin{remark}
  We will write $g_{x}$ instead of $\eta_{X}(x)$, that is,  $g_{x}:=\eta_{X}(x)$. Hence, from the definition of $As(X)$, it has the presentation
\begin{align*}
    As(X)= \langle g_{x} , \ x \in X  \ | \ g_{x\vartriangleright y} = g_{x}g_{y}g_{x}^{-1} , \ x,y \in X \rangle. 
\end{align*}
Let $ z =  x \vartriangleright^{-1} y  \in X$, then $x \vartriangleright z = y$, thus we have $g_{y} = g_{x \vartriangleright z}= g_{x}g_{z}g_{x}^{-1}$. Thereby, we also have the relation $g_{x \vartriangleright^{-1} y} = g_{x}^{-1}g_{y}g_{x}$.
\end{remark}

In, \cite{Nosaka} and \cite{Joyce2}, there are two functors that are studied: $As$, from the category of racks into the category of groups, and $Conj$, from the category of groups into the category of racks. Moreover, it is proved that the functor $As$ is the left adjoint of the functor $Conj$. Namely, for any rack $X$ and any group $G$ we have

\begin{center}
    $Hom_{gr}(As(X),G) \cong Hom_{rack}(X, Conj(G))$.
\end{center}

This claim is proved in the following theorem.
	 
 \begin{theo}\label{Theorem 1.2} \textnormal{\textbf{(Universal property)}}\\
 Let $G$ be a group, $X$ be a rack and $\phi: X \longrightarrow Conj(G)$ be a rack homomorphism.  Then $\phi$ induces an unique group homomorphism (lifting of $\phi$) $\hat{\phi}: As(X) \longrightarrow G$ such that $\hat{\phi}(g_{x})=\phi(x)$ for all $x \in X$. Conversely, let $\psi: As(X) \longrightarrow G$ be a group homomorphism then  $\psi$ induces a rack homomorphism (projection of $\psi$) $\hat{\psi}: X \longrightarrow Conj(G)$ such that $\hat{\psi}(x)=\psi(g_{x})$ for all $x \in X$.
 \end{theo}
 
 \begin{proof} Let $\phi: X \longrightarrow Conj(G)$ be a rack homomorphism. Since $G= Conj(G)$ as sets, then the function $\phi: X \longrightarrow G$ makes sense. From the universal property of free groups, there exists an unique group homomorphism $f_{\phi}: F(X) \longrightarrow G$ such that $f_{\phi} (x) = \phi(x)$ for all $x \in X$. In $G$, we have that $\phi(x) \ast \phi(z)= \phi(x)\phi(z)\phi(x)^{-1} $, for all $x,z \in X$. Hence, 
  \begin{align*}
    f_{\phi}((x \vartriangleright z)xz^{-1}x^{-1}) &= f_{\phi}(x \vartriangleright z)f_{\phi}(x)[f_{\phi}(z)]^{-1}[f_{\phi}(x)]^{-1}\\
    &=\phi(x \vartriangleright z)\phi(x)[\phi(z)]^{-1}[\phi(x)]^{-1}\\
    &=[\phi(x) \ast \phi(z)]\phi(x)[\phi(z)]^{-1}[\phi(x)]^{-1}\\
    &=1.
  \end{align*}
 This implies that $(x \vartriangleright z)xz^{-1}x^{-1} \in Ker f_{\phi}$, for all $x,z \in X$. It follows that $f_{\phi}$ defines an unique group homomorphism $\hat{\phi}: As(X) \longrightarrow G$ such that $\hat{\phi}(g_{x}) = f_{\phi}(x)= \phi(x)$ for all $x \in X$. 
 
 Conversely, let $\psi: As(X) \longrightarrow G$ be a group homomorphism. Because $g_{x\vartriangleright z}g_{x}g_{y}^{-1}g_{x}^{-1}=1$, in $As(X)$, then for every $x,z \in X$,
  \[
     \psi(1) = \psi(g_{x\vartriangleright z}g_{x}g_{y}^{-1}g_{x}^{-1})=  \psi(g_{x \vartriangleright z}) \psi(g_{x})\psi(g_{z})^{-1}\psi(g_{x})^{-1}=1.
 \]
Therefore, $\psi(g_{x \vartriangleright z}) = \psi(g_{x})\psi(g_{z})\psi(g_{x})^{-1} .$ 
Consider the natural map $\eta_{X}: X \longrightarrow As(X) $, then the map $\hat{\psi}: X \longrightarrow Conj(G)$ defined by $\hat{\psi}(x)= (\psi \circ \eta)(x)= \psi(g_{x}) $ is a rack homomorphism. Indeed, $\hat{\psi}(x \vartriangleright z) = \psi(g_{x \vartriangleright z}) =\psi(g_{x})\psi(g_{z})\psi(g_{x})^{-1} = \hat{\psi}(x)\hat{\psi}(y)\hat{\psi}(x)^{-1}$.
\end{proof}
In \cite[Remark~2.14, page 8]{Ven} Vendramin gives an useful hint for the proof of the following proposition, so we omit it. 
\begin{prop}\label{asinfi}
The associated group $As(X)$ of a rack $X$ is an infinite group.
\end{prop}

We review another group associated to a rack, called the \textit{inner automorphism group}.

\begin{defi}
The \textbf{\textit{inner automorphism group}}, denoted by $Inn(X)$, is defined as the subgroup of $Aut(X)$, generated by the permutation functions $L_{x}$. Concisely, the group is
  \begin{align*}
     Inn(X):= \langle L_{x} \ | \ x\in X \rangle. 
  \end{align*}

\end{defi}

Note that, for a rack $X$, the map $L: X \longrightarrow Conj(Inn(X))$ defined by $L(x):= L_{x}$, is a rack homomorphism. Hence, from Theorem \ref{Theorem 1.2} there exits a group homomorphism $\hat{L}: As(X) \longrightarrow Inn(X)$ defined by $\hat{L}(g_{x}):=L_{x}$ for all $x \in X$. Moreover, this group homomorphism is surjective, then from first isomorphism theorem $As(X)/ ker \hat{L} \cong Inn(X)$.\\
There is a natural action of the inner automorphism group $Inn(X)$ of a rack $X$, over the underlying set $X$, given by the function
\begin{align*}
    \bullet : Inn(X) \times X &\longrightarrow X\\
             (L_{x}, y) &\longmapsto L_{x}\bullet y := L_{x}(y)= x\vartriangleright y,
  \end{align*}
The orbits of this action are known as \textbf{\textit{the connected components}} of $X$. A rack $X$ is said to be \textbf{\textit{connected}} (or \textit{indecomposable}) if the action of $Inn(X)$ on $X$ is transitive, that means, if it has only one orbit. Thus, given  $y \in X$ there exists $\phi \in Inn(X)$ such that $\phi(x) = y$. Since $\phi \in Inn(X)$ then it has the form $\phi = L_{x_{n}}^{\epsilon_{n}}L_{x_{n-1}}^{\epsilon_{n-1}} \cdots L_{x_{1}}^{\epsilon_{1}}$,  where $x_{i} \in X$ and $\epsilon_{i} \in \{1,-1\}$, for all $i \in \{1,\dots,n\}$. So, we have
  {\small{$$
      y = \phi (x)=L_{x_{n}}^{\epsilon_{n}}L_{x_{n-1}}^{\epsilon_{n-1}} \cdots L_{x_{1}}^{\epsilon_{1}}(x) =x_{n} \vartriangleright^{\epsilon_{n}} (x_{n-1} \vartriangleright^{\epsilon_{n-1}} ( \cdots \vartriangleright^{\epsilon_{2}} ( x_{1} \vartriangleright^{\epsilon_{1}} x) \cdots ). 
  $$}}

In other words, a rack $X$ is connected if and only if for all $x,y \in X$ there are $x_{1},\dots,x_{n} \in X$ such that  $L_{x_{n}}^{\epsilon_{n}}L_{x_{n-1}}^{\epsilon_{n-1}} \cdots L_{x_{1}}^{\epsilon_{1}}(x) =y$, where $\epsilon_{i} \in \{1,-1\}$, for all $i \in \{1,\dots,n\}$.

In the following example we illustrate the construction of the associated group and of the inner automorphism of a certain type of racks. 

\begin{eg} \label{Example 1.12}
Let $X=\{1,2,3,\dots,n\}$ and $\sigma  \in \mathbb{S}_{n}$. The set $X$ is a rack with the binary operation $ \vartriangleright: X \times X \longrightarrow X$, defined by $i \vartriangleright j = \sigma(j)$, for all $i,j \in X$. This rack is known as a \textit{type cyclic rack}. Indeed, note that $L_{i} = \sigma$, for all $i \in X$, then, the inner automorphism group $Inn(X) = \langle L_{1} \rangle \cong \z_{m}$, where $m$ is the order of $\sigma$. Note that, if $\sigma = (1 \ 2\ \cdots\ n) \in \mathbb{S}_{n}$, we have that $g_{k}=g_{1\vartriangleright (k-1)}= g_{1}g_{k-1}g_{1}^{-1}$, for all $k=2,\dots, n$. Thereby, we can conclude that  $g_{k} = g_{1}$, for all $k=2,\dots,n$. Then $As(X) \cong \z$.

\end{eg}
	 
Graña et al, in \cite{finitenveloping}, introduce another group associated to finite connected racks. They name it, the \textit{finite enveloping group}, and as the name suggests, this group is finite. We follow the construction given by them, but for the finiteness we give a slightly different proof.

\begin{Notation}
 Let $X$ be a rack, $x,y \in X$ and $k \in \n$; we write
 \begin{align*}
     x \vartriangleright^{k} y = L_{x}^{k}(y) = x \vartriangleright(x \vartriangleright( \cdots \vartriangleright(x \vartriangleright y)) \cdots ) \ , \hspace{0.5cm} \textit{x multiplying k times}.
 \end{align*}
\end{Notation}

The proof of the following result can be gotten in the same fashion as the one given in \cite[Lemma 2.1]{finitenveloping} by considering the identity $L_{x \vartriangleright y} = L_{x}L_{y}L_{x}^{-1} $, for all $x,y \in X$. So we omit it.

\begin{lemma} \label{Lemma 1.2}
Let $X$ be a rack and $n \in \n$ . Let $x_{1},\dots,x_{n}, y,z \in X$, such that 
\begin{align*}
    L^{\epsilon_{n}}_{x_{n}}L^{\epsilon_{n-1}}_{n-1} \cdots L^{\epsilon_{1}}_{x_{1}} (z) = y.  
\end{align*}
where $\epsilon_{i} = \pm 1 $, for all $i \in \{1,\dots,n\}$. Then, $L^{\epsilon_{n}}_{x_{n}} \cdots L^{\epsilon_{1}}_{x_{1}} L_{z} = L_{y} L^{\epsilon_{n}}_{x_{n}} \cdots L^{\epsilon_{1}}_{x_{1}}.$  
\end{lemma}

\begin{lemma} \cite[Lemma 2.16]{finitenveloping}\label{Lemma 1.3}
Let $X$ be a rack and $n \in \n$, then we have the following relations in the associated group $As(X)$: 
\begin{itemize}
    \item [(1)] $g_{x}^{n}g_{y}= g_{x \vartriangleright^{n} y}g_{x}^{n}$, \ $\forall x,y \in X$.
    \item[(2)] $g_{x}g_{y}^{n} = g_{x \vartriangleright y}^{n}g_{x}$, \ $\forall x,y \in X$.
\end{itemize}
\end{lemma}

\begin{prop}\cite[Lemma 2.17]{finitenveloping}\label{proposition 1.3}
    Let $X$ be a rack and $x \in X$ such that $L_{x}$ has finite order $n$, then in the associated group $As(X)$,  we have that $g_{x}^{n} \in Z(As(X))$, where $Z(As(X))$ is the center of the group $As(X)$.
\end{prop}

\begin{theo}\cite[Lemma 2.18]{finitenveloping} \label{Theorem 1.4}
    Let $X$ be a finite connected rack. Then for every $x \in X$, the permutations $L_{x} \in Inn(x)$  have the same order. Furthermore, if $n$ is the order of all the permutations $L_{x}$ then, in the associated group $As(X)$, we have the relation $g_{x}^{n} = g_{y}^{n}$ for all $x,y \in X$.
\end{theo}

Let $X$ be a finite and connected rack, let $x_{0} \in X$ and $n$ be the order of the function $L_{x_{0}}$. Observe that, from the previous theorem and Proposition \ref{proposition 1.3}, the group $ \langle g_{x_{0}}^{n} \rangle \subset As(X)$, generated by $g_{x_{0}}^{n}$, is a normal subgroup of $As(X)$. Therefore, we can consider the quotient group  $As(X)/ \langle g_{x_{0}}^{n} \rangle$.

\begin{defi}
Let $X$ be a finite and connected rack. The quotient group $As(X)/ \langle g_{x_{0}}^{n} \rangle$, denoted by $\overline{G_{X}}$, is called the \textit{\textbf{finite enveloping group}}.
\end{defi}

With the above notation and from the definition of $As(X)$ we have a presentation for the group $\overline{G_{X}}$.
\begin{align*}
    \overline{G_{X}}= \langle g_{x} \ , \ x \in X  \ | \ g_{x\vartriangleright y} = g_{x}g_{y}g_{x}^{-1} , g_{x}^{n} =1,  \ x,y \in X \rangle. 
\end{align*}

The  following theorem can be found in \cite[Lemma 2.19]{finitenveloping}, here we give a more combinatorial proof.
\begin{theo}\label{Theorem 1.5}
    Let $X$ be a finite and connected rack, then the group $\overline{G_{X}}$ is finite.
\end{theo}

\begin{proof} Let $X = \{x_{0},\dots,x_{k-1}\}$ be a connected rack. Suppose that $n \in \n$ is the order of $L_{x_{i}}$ for all $i \in \{0,\dots,k-1\} $. From Theorem \ref{Theorem 1.4}, for every $i \in \{0,\dots,k-1\}$, we have that $g_{x_{i}}^{n} \langle g_{x_{0}}^{n} \rangle =\langle g_{x_{0}}^{n} \rangle$, thus, $g_{x_{i}}^{-1}\langle g_{x_{0}}^{n} \rangle = g_{x_{i}}^{n-1}\langle g_{x_{0}}^{n} \rangle$.\\
We claim that every word $\omega \in \overline{G_{X}}$ is of the form $\omega =g_{x_{i_{1}}}^{u_{1}} g_{x_{i_{2}}}^{u_{2}}\cdots g_{x_{i_{m}}}^{u_{m}}\langle g_{x_{0}}^{n} \rangle$ where $m \leq k$, $u_{t} \in \{0,1,\dots,n-1\}$, $x_{i_{t}} \in X$ for all $t \in \{1,\dots,m\}$ and $g_{x_{i_{r}}} \neq g_{x_{i_{s}}}$, for every $r \neq s$. In fact, let $\omega \in \overline{G_{X}}$ thus, from the definition of $\overline{G_{X}}$, $\omega = g_{x_{j_{1}}}^{e_{1}}g_{x_{j_{2}}}^{e_{2}} \cdots g_{x_{j_{\hat{m}}}}^{e_{\hat{m}}} \langle g_{x_{0}}^{n} \rangle$, where $e_{t} \in \{0,1,\dots,n-1\}$. Suppose that there exists $r < s$, such that $g_{x_{j_{r}}} = g_{x_{j_{s}}}$, and for every $t \in \{r+1,\dots,s-1\}$ it satisfies that $g_{x_{j_{t}}} \neq g_{x_{j_{s}}}$. Therefore, we have that $\omega$ has the following form
\begin{center}
    $ (g_{x_{j_{1}}}^{e_{1}}\cdots g_{x_{j_{r-1}}}^{e_{r-1}}\langle g_{x_{0}}^{n} \rangle)( g_{x_{j_{r}}}^{e_{r}}g_{x_{j_{r+1}}}^{e_{r+1}} \cdots g_{x_{j_{s-1}}}^{e_{s-1}}g_{x_{j_{s}}}^{e_{s}}\langle g_{x_{0}}^{n} \rangle)(g_{x_{j_{s+1}}}^{e_{s+1}}g_{x_{j_{s+2}}}^{e_{s+2}}\cdots g_{x_{j_{\bar{m}}}}^{e_{\bar{m}}}\langle g_{x_{0}}^{n} \rangle)$.
\end{center}
First, we prove that in the associated group $As(X)$ we have
\begin{align*}
    g_{x_{j_{r}}}^{e_{r}}g_{x_{j_{r+1}}}^{e_{r+1}} \cdots g_{x_{j_{s-1}}}^{e_{s-1}}g_{x_{j_{s}}}^{s}
    &= g_{x_{j_{r}} \vartriangleright^{e_{r}} x_{j_{r+1}}}^{e_{r+1}} \cdots g_{x_{j_{r}} \vartriangleright^{e_{r}} x_{j_{s-1}}}^{e_{s-1}}g_{x_{j_{s}}}^{e_{s} +e_{r}}.
\end{align*}
We use induction on $e_{r}$. Let $e_{r}=1$ then, from Lemma \ref{Lemma 1.3} $(2)$, we have that $g_{x_{j_{r}}}g_{x_{j_{t}}}^{e_{t}} = g_{x_{j_{r}} \vartriangleright x_{j_{t}}}^{e_{t}} g_{x_{j_{r}}}$, for every $t \in \{r+1,\dots,s-1\}$. It follows that,
\begin{align*}
    g_{x_{j_{r}}}g_{x_{j_{r+1}}}^{e_{r+1}} \cdots g_{x_{j_{s-1}}}^{e_{s-1}}g_{x_{j_{s}}}^{e_{s}} 
    &= g_{x_{j_{r}} \vartriangleright x_{j_{r+1}}}^{e_{r+1}}g_{x_{j_{r}}}g_{x_{j_{r+2}}}^{e_{r+2}} \cdots g_{x_{j_{s}}}^{e_{s}}\\
    &= g_{x_{j_{r}} \vartriangleright x_{j_{r+1}}}^{e_{r+1}}g_{x_{j_{r}} \vartriangleright x_{j_{r+2}}}^{e_{r+2}}g_{x_{j_{r}}} \cdots g_{x_{j_{s}}}^{e_{s}}\\
    &\ \ \vdots\\
    &= g_{x_{j_{r}} \vartriangleright^{e_{r}} x_{j_{r+1}}}^{e_{r+1}} \cdots g_{x_{j_{r}} \vartriangleright^{e_{r}} x_{j_{s-1}}}^{e_{s-1}}g_{x_{r}}g_{x_{j_{s}}}^{e_{s} }\\
    &= g_{x_{j_{r}} \vartriangleright^{e_{r}} x_{j_{r+1}}}^{e_{r+1}} \cdots g_{x_{j_{r}} \vartriangleright^{e_{r}} x_{j_{s-1}}}^{e_{s-1}}g_{x_{j_{s}}}^{e_{s} +1}.
\end{align*}
Hence,  the result holds for $e_{r}=1$. Now, suppose that
\begin{center}
    $ g_{x_{j_{r}}}^{e_{r}-1}g_{x_{j_{r+1}}}^{e_{r+1}} \cdots g_{x_{j_{s-1}}}^{e_{s-1}}g_{x_{j_{s}}}^{s}
    = g_{x_{j_{r}} \vartriangleright^{e_{r}-1} x_{j_{r+1}}}^{e_{r+1}} \cdots g_{x_{j_{r}} \vartriangleright^{e_{r}-1} x_{j_{s-1}}}^{e_{s-1}}g_{x_{j_{s}}}^{e_{s} +e_{r}-1}$.
\end{center}
 Then, from Lemma \ref{Lemma 1.3} $(2)$, 
\begin{align*}
    g_{x_{j_{r}}}^{e_{r}}g_{x_{j_{r+1}}}^{e_{r+1}} \cdots g_{x_{j_{s-1}}}^{e_{s-1}}g_{x_{j_{s}}}^{s}
    &=  g_{x_{j_{r}}}[g_{x_{j_{r}}}^{e_{r}-1}g_{x_{j_{r+1}}}^{e_{r+1}} \cdots g_{x_{j_{s-1}}}^{e_{s-1}}g_{x_{j_{s}}}^{s}]\\    
    &=g_{x_{j_{r}}}[g_{x_{j_{r}} \vartriangleright^{e_{r}-1} x_{j_{r+1}}}^{e_{r+1}} \cdots g_{x_{j_{r}} \vartriangleright^{e_{r}-1} x_{j_{s-1}}}^{e_{s-1}}g_{x_{j_{s}}}^{e_{s} +e_{r}-1}]\\
    &=g_{x_{j_{r}} \vartriangleright^{e_{r}} x_{j_{r+1}}}^{e_{r+1}} g_{x_{j_{r}}}g_{x_{j_{r}} \vartriangleright^{e_{r}-1} x_{j_{r+2}}}^{e_{r+2}}\cdots g_{x_{j_{r}} \vartriangleright^{e_{r}-1} x_{j_{s-1}}}^{e_{s-1}}g_{x_{j_{s}}}^{e_{s} +e_{r}-1}\\
    &=g_{x_{j_{r}} \vartriangleright^{e_{r}} x_{j_{r+1}}}^{e_{r+1}}g_{x_{j_{r}} \vartriangleright^{e_{r}} x_{j_{r+2}}}^{e_{r+2}} g_{x_{j_{r}}}\cdots g_{x_{j_{r}} \vartriangleright^{e_{r}-1} x_{j_{s-1}}}^{e_{s-1}}g_{x_{j_{s}}}^{e_{s} +e_{r}-1}\\
    &\ \ \vdots \\
    &=g_{x_{j_{r}} \vartriangleright^{e_{r}} x_{j_{r+1}}}^{e_{r+1}}g_{x_{j_{r}} \vartriangleright^{e_{r}} x_{j_{r+2}}}^{e_{r+2}} \cdots g_{x_{j_{r}} \vartriangleright^{e_{r}} x_{j_{s-1}}}^{e_{s-1}}g_{x_{j_{r}}}g_{x_{j_{s}}}^{e_{s} +e_{r}-1}\\
    &=g_{x_{j_{r}} \vartriangleright^{e_{r}} x_{j_{r+1}}}^{e_{r+1}}g_{x_{j_{r}} \vartriangleright^{e_{r}} x_{j_{r+2}}}^{e_{r+2}} \cdots g_{x_{j_{r}} \vartriangleright^{e_{r}} x_{j_{s-1}}}^{e_{s-1}}g_{x_{j_{s}}}^{e_{s} +e_{r}}.
\end{align*}
And the result follows.

Now, observe that $g_{x_{j_{r}} \vartriangleright^{e_{r}} x_{j_{t}}} \neq g_{x_{j_{r}}}$, for every $t \in \{r+1,\dots,s-1\}$. Indeed, suppose that $g_{x_{j_{r}} \vartriangleright^{e_{r}} x_{j_{t}} }= g_{x_{j_{r}}}$ for some $t \in \{r+1,\dots,s-1\}$. From the definition of $As(X)$,
\begin{align*}
    g_{x_{j_{r}}} &= g_{x_{j_{r}} \vartriangleright^{e_{r}} x_{j_{t}}}\\
    &= g_{x_{j_{r}} \vartriangleright (x_{j_{r}} \vartriangleright^{e_{r}-1} x_{j_{t}})}\\
    &= g_{x_{j_{r}}}g_{x_{j_{r}} \vartriangleright^{e_{r}-1} x_{j_{t}}}g_{x_{j_{r}}}^{-1}\\
     &\ \ \vdots\\
    &= g_{x_{j_{r}}}^{e_{r}} g_{x_{j_{t}}}g_{x_{j_{r}}}^{-e_{r}}.
\end{align*}
Therefore, we have  $g_{x_{j_{t}}}= g_{x_{j_{r}}}^{e_{r}} g_{x_{j_{r}}}g_{x_{j_{r}}}^{-e_{r}} = g_{x_{j_{r}}}^{e_{r}-e_{r}+1} = g_{x_{j_{r}}}=g_{x_{j_{s}}}$, which contradicts the assumption that for every $t \in \{r+1,\dots,s-1\}$, $g_{x_{j_{t}}} \neq g_{x_{j_{s}}}$. Thus, $g_{x_{j_{r}} \vartriangleright^{e_{r}} x_{j_{t}}} \neq g_{x_{j_{r}}}$ for every $t \in \{r+1,\dots,s-1\}$.

From the above, we can write the word $\omega $ as
{\footnotesize{
\begin{align*}
    \omega &=  (g_{x_{j_{1}}}^{e_{1}}\cdots g_{x_{j_{r-1}}}^{e_{r-1}}\langle g_{x_{0}}^{n} \rangle)( g_{x_{j_{r}}}^{e_{r}} \cdots g_{x_{j_{s-1}}}^{e_{s-1}}g_{x_{j_{s}}}^{e_{s}}\langle g_{x_{0}}^{n} \rangle)(g_{x_{j_{s+1}}}^{e_{s+1}}\cdots g_{x_{j_{\bar{m}}}}^{e_{\bar{m}}}\langle g_{x_{0}}^{n} \rangle)\\
    &= (g_{x_{j_{1}}}^{e_{1}}\cdots g_{x_{j_{r-1}}}^{e_{r-1}}\langle g_{x_{0}}^{n} \rangle)(g_{x_{j_{r}} \vartriangleright^{e_{r}} x_{j_{r+1}}}^{e_{r+1}} \cdots g_{x_{j_{r}} \vartriangleright^{e_{r}} x_{j_{s-1}}}^{e_{s-1}}g_{x_{j_{s}}}^{e_{s} +e_{r}}\langle g_{x_{0}}^{n} \rangle)(g_{x_{j_{s+1}}}^{e_{s+1}}\cdots g_{x_{j_{\bar{m}}}}^{e_{\bar{m}}}\langle g_{x_{0}}^{n} \rangle)\\
    &=g_{x_{j_{1}}}^{e_{1}}\cdots g_{x_{j_{r-1}}}^{e_{r-1}}g_{x_{j_{r}} \vartriangleright^{e_{r}} x_{j_{r+1}}}^{e_{r+1}} \cdots g_{x_{j_{r}} \vartriangleright^{e_{r}} x_{j_{s-1}}}^{e_{s-1}}g_{x_{j_{s}}}^{e_{s} +e_{r}}g_{x_{j_{s+1}}}^{e_{s+1}}\cdots g_{x_{j_{\bar{m}}}}^{e_{\bar{m}}}\langle g_{x_{0}}^{n} \rangle\\
    &= g_{x_{i_{1}}}^{u_{1}} \cdots g_{x_{i_{m}}}^{u_{m}}\langle g_{x_{0}}^{n} \rangle.
\end{align*}}}

From where, $g_{x_{i_{r}}} \neq g_{x_{i_{s}}}$, for every $r \neq s$. If $m > k$. Then, we would have repeated occurrences and we can apply the same process shown above to reduce the word. Therefore, every word $\omega \in \overline{G_{X}}$ is of the form $\omega =g_{x_{i_{1}}}^{u_{1}} \cdots g_{x_{i_{m}}}^{u_{m}}\langle g_{x_{0}}^{n} \rangle$, where $m \leq k$, $u_{t} \in \{0,1,\dots,n-1\}$, $x_{i_{t}} \in X$ for all $t \in \{1,\dots,m\}$ and $g_{x_{i_{r}}} \neq g_{x_{i_{s}}}$ for every $r \neq s$. It follows that $\overline{G_{X}}$ is finite.
\end{proof}

\section{The Permutation Quandle}\label{NewSection 1.3}
Vendramin developed a GAP package \cite{vendramin2012classification}, which is able to compute the finite enveloping group of connected quandles of order less than 48. Based on this, we have identified the finite enveloping group of a family of quandles that we refer to as the \textit{permutation quandle}. Let us consider the following lemma.

\begin{lemma}\label{Lemma 1.1}
Let $G$ be a group. Consider the quandle $Conj (G)$, then every conjugacy class of $G$ is a subquandle of $Conj(G)$.
\end{lemma}

\begin{proof} Let $g \in G$ and let $C_{g} = \{hgh^{-1} \ | \ h \in G\}$ be the conjugacy class represented by $g$. Let $q,p \in C_{g}$, we have to prove that $q \vartriangleright p \in C_{g}$. First, suppose that $q = p$, then $q \vartriangleright p = p \vartriangleright p = ppp^{-1}=p \in C_{g}$.\\
Now, suppose that $q \neq p$. Since $q,p \in C_{g}$ then there exists $h_{1},h_{2} \in G$ such that $p= h_{1}gh^{-1}_{1} = h_{1} \vartriangleright g$ and $q = h_{2}gh^{-1}_{2} = h_{2} \vartriangleright g$. We have,
\begin{align*}
     q \vartriangleright p &= qpq^{-1}=(h_{2}gh^{-1}_{2})(h_{1}gh^{-1}_{1})(h_{2}gh^{-1}_{2})^{-1}\\
    &=(h_{2}gh^{-1}_{2}h_{1})g(h^{-1}_{1}h_{2}g^{-1}h^{-1}_{2})= (h_{2}gh^{-1}_{2}h_{1})g(h_{2}gh^{-1}_{2}h_{1})^{-1}.
\end{align*}
Therefore, $\ q \vartriangleright p \in C_{g}$. Thus, $C_{g}$ is a subquandle of $Conj(G)$.
 \end{proof}

It is well known that the set of all transpositions of the symmetric group $\mathbb{S}_{n}$, forms a conjugacy class of $\mathbb{S}_{n}$. Thus, from Lemma \ref{Lemma 1.1}, that set is a subquandle of the quandle $Conj(\mathbb{S}_{n})$, and then, the following definition makes sense.   

  \begin{defi}[The permutation quandle]\label{Definition 1.7}
Consider the conjugacy quandle $Conj(\mathbb{S}_{n})$ of the symmetric group $\mathbb{S}_{n}$ with $n \geq 3$. We define the \textit{\textbf{permutation quandle}}, denoted by $\mathbb{P}_{n}$, as the subquande of $Conj(\mathbb{S}_{n})$ of all transpositions of $\mathbb{S}_{n}$.
\end{defi}

That is, the permutation quandle $\mathbb{P}_{n}$ is the set  $\mathbb{P}_{n} := \{(i \ j) \ | \  (i \ j) \in \mathbb{S}_{n}\}$, with the operation 
\[(i \ j) \vartriangleright (k \ r) =  (i \ j)(k \ r)(i \ j)^{-1} = (i \ j)(k \ r)(i \ j),\]
where $(i \ j)$ denotes the transposition that interchanges $i$ and $j$. The name of the permutation quandle will be justified later on. For now, note that for every $(i \ j),(k \ r) \in \mathbb{P}_{n}$, we have
\begin{align*}
        (i \ j) \vartriangleright [ (i \ j) \vartriangleright (k \ r)] &= (i \ j) \vartriangleright [ (i \ j)(k \ r)(i \ j)]\\
       &= (i \ j)[(i \ j)(k \ r)(i \ j)](i \ j)\\
    &= (k \ r). 
\end{align*}
So, $\mathbb{P}_{n}$ is an \textit{involutive} quandle.

\begin{eg} \label{Example 1.8}
Consider the permutation quandle $\mathbb{P}_{3} = \{(i \ j) \ | \ (i \ j) \in \mathbb{S}_{3}\}$. That is, $\mathbb{P}_{3} = \{ (1 \ 2), (1 \ 3),(2 \ 3)\}$. Then,

\begin{center}
\begin{tabular}{c|c|c|c}
$\vartriangleright$ & \textbf{(2 \ 3)} & \textbf{(1 \ 3)}  & \textbf{(1 \ 2)}\\ \hline
$\textbf{(2 \ 3)}$ & (2 \ 3) & (1 \ 2) & (1 \ 3) \\ \hline
$\textbf{(1 \ 3)}$ & (1 \ 2) & (1 \ 3) & (2 \ 3) \\ \hline
$\textbf{(1 \ 2)}$ & (1 \ 3) & (2 \ 3) & (1 \ 2) 
\end{tabular}
\end{center}
\end{eg}

\begin{lemma}
The permutation quandle $\mathbb{P}_{n}$ is a connected quandle.
\end{lemma}

\begin{proof} Let $(i \ j), (l \ t) \in \mathbb{P}_{n}$. Since $(i \ j) = (j \ i)$, for all $(i \ j) \in \mathbb{P}_{n}$, then without loss of generality, we have three cases
\begin{itemize}
    \item [(I)] If $(i \ j) = (l \ t)$ then $L_{(i \ j)}[(l \ t)] = (i \ j) \vartriangleright (l \ t) = (i \ j)(l \ t)(i \ j) = (i \ j)$.
    \item [(II)]Suppose that $i \neq l$ and $j = t$, then 
    $$L_{(i \ l)}[(l \ t)] = (i \ l) \vartriangleright (l \ t) =  (i \ l)(l \ t)(i \ l)= (i \ t) = (i \ j).$$ 
    \item[(III)]  If $i \neq l$ and $j \neq  t$ then 
    $$
    \begin{aligned}
     L_{(j \ t)}L_{(i \ l)}[(l \ t)] &= (j \ t) \vartriangleright [ (i \ l) \vartriangleright (l \ t)] = (j \ t) \vartriangleright[(i \ l)(l \ t)(i \ l)]\\
     &= (j \ t) \vartriangleright (i \ t) =(j \ t)(i \ t)(j \ t)= (i \ j).   
    \end{aligned}
    $$
\end{itemize}

Therefore the permutation quandle $\mathbb{P}_{n}$ is connected.
\end{proof}

The next proposition provides us a complete description of the finite enveloping group $\overline{G_{\mathbb{P}_{n}}}$ of the permutation quandle $\mathbb{P}_{n}$. Specifically, we prove that $\overline{G_{\mathbb{P}_{n}}} \cong \mathbb{S}_{n}$, which justifies the name   ``\textit{permutation quandle}''.

\begin{prop}\label{proposition 1.4}
Let $\mathbb{P}_{n}$ be the permutation quandle then its finite enveloping group $\overline{G_{\mathbb{P}_{n}}}$ is isomorphic to the symmetric group $\mathbb{S}_{n}$.
\end{prop}

\begin{proof} Consider the function $\psi: \mathbb{P}_{n} \longrightarrow Conj(\mathbb{S}_{n})$, defined by $\psi[(i \ j)]:= (i \ j)$ for all $(i \ j) \in \mathbb{P}_{n}$. Since $\mathbb{P}_{n}$ is a subquandle of $Conj(\mathbb{S}_{n})$, then for every $(i \ j), (k \ r) \in \mathbb{P}_{n} $ we have that $(i \ j) \vartriangleright (k \ r) \in \mathbb{P}_{n} $. That is, there exists $(l \ t) \in \mathbb{P}_{n} \subset \mathbb{S}_{n}$, such that $(i \ j) \vartriangleright (k \ r) = (l \ t)$, then
\begin{align*}
    \psi[(i \ j) \vartriangleright (k \ r)] &= \psi[(l \ t)] = (l \ t) =(i \ j) \vartriangleright (k \ r)
    = \psi[(i \ j)] \vartriangleright \psi[(k \ r)].
\end{align*}
It implies that $\psi$ is a quandle homomorphism. Now, from Theorem \ref{Theorem 1.2} there exists a group homomorphism $\hat{\psi}: As(\mathbb{P}_{n}) \longrightarrow \mathbb{S}_{n}$ such that $\hat{\psi}(g_{(i  j)})= \psi[(i \ j)]$. Since the permutation quandle is involutive and connected, then $L_{(i \ j)}^{2} =id$, for all $(i \ j) \in \mathbb{P}_{n}$. Which implies that $g_{(ij)}^{2}=g_{(k r)}^{2}$, for all $(i \ j), (k \ r) \in \mathbb{P}_{n}$. Observe that,
\begin{center}
    $\hat{\psi}(g_{(1 2)}^{2}) = \hat{\psi}^{2}(g_{(1 2)}) = (1 \ 2)^{2} = 1_{\mathbb{S}_{n}}$.
\end{center}
 Therefore, $\langle g_{(12)}^{2} \rangle \subset ker(\hat{\psi})$. Thus, $\hat{\psi}$ induces a group homomorphism $\bar{\psi}: \overline{G_{\mathbb{P}_{n}}} \longrightarrow \mathbb{S}_{n}$ such that $\bar{\psi}(g_{(ij)}\langle g_{(12)}^{2} \rangle) = \hat{\psi}(g_{(ij)}) = \psi[(i \ j)] = (i \ j)$. Let us see that $\bar{\psi}$ is bijective. In fact, let $\sigma \in \mathbb{S}_{n}$, since the set of all transpositions is a generating set of $\mathbb{S}_{n}$ then there exist $(i_{1} \ j_{1}),\dots, (i_{k} \ j_{k}) \in \mathbb{S}_{n}$ such that $\sigma = (i_{1} \ j_{1})(i_{2} \ j_{2}) \cdots (i_{k} \ j_{k})$. Therefore,
\begin{align*}
    \sigma &= (i_{1} \ j_{1})(i_{2} \ j_{2}) \cdots (i_{k} \ j_{k})\\
    &= \bar{\psi}(g_{(i_{1}j_{1})}\langle g_{(12)}^{2} \rangle) \bar{\psi}(g_{(i_{2}j_{2})}\langle g_{(12)}^{2} \rangle)\cdots \bar{\psi}(g_{(i_{k}j_{k})}\langle g_{(12)}^{2} \rangle) \\
    &= \bar{\psi}(g_{(i_{1}j_{1})}g_{(i_{2}j_{2})} \cdots g_{(i_{k}j_{k})}\langle g_{(12)}^{2} \rangle).
\end{align*}
Hence, $\bar{\psi}$ is surjerctive.\\
Now, we use the presentation of the symmetric group $\mathbb{S}_{n} = \langle \sigma_{i},\dots, \sigma_{n-1} \ | \ \sigma^{2}_{i} = 1, (\sigma_{i} \sigma_{i+1})^{3} = 1,(\sigma_{i} \sigma_{j})^{2} = 1 , |j-1|>1 \ \rangle $, where $\sigma_{i} = (i \ \ i+1)$ for all $i \in \{1,\dots ,n-1\}$ to prove the injectivity. \\
Let $(k \ r) \in \mathbb{P}_{n}$, since $(k \ r ) = (r \ k)$, without loss of generality, we suppose that $k <r$. Therefore, $r-k >0$. If $r-k = 1$, then $r = k+1$ and we have that $(k \ r) = (k \ \ k+1)$. If $r-k > 1$, note that,
\begin{align*}
    (k \ r) &= (k \ \ k+1)(k+1 \ \ r)(k \ \ k+1)\\ 
    &= (k \ \ k+1) \vartriangleright (k+1 \ \ r),\\
    \\
    (k+1 \ \ r) &= (k+1 \ \ k+2)(k+2 \ \ r)(k+1 \ \ k+2)\\
    &= (k+1 \ \ k+2) \vartriangleright (k+2 \ \ r),\\
    & \ \ \vdots\\
   (r-2 \ \ r) &= (r-2 \ \ r-1)(r-1 \ \ r)(r-2 \ \ r-1)\\
    &= (r-2 \ \ r-1) \vartriangleright (r-1 \ \ r).
\end{align*}
Therefore, 
$$
    (k \ r) = ( k \ \ k+1) \vartriangleright [(k+1 \ \ k+2) \vartriangleright [ \cdots \vartriangleright [(r-2 \ \ r-1) \vartriangleright (r-1 \ \ r)] \cdots].
$$
From where, it follows that
\begin{align*}
    g_{(k \ r)} &= g_{( k \  k+1) \vartriangleright [(k+1 \  k+2) \vartriangleright [ \cdots \vartriangleright [(r-2 \  r-1) \vartriangleright (r-1 \  r)] \cdots]}\\
    &= g_{( k \  k+1)}g_{(k+1 \  k+2) \vartriangleright [ \cdots \vartriangleright [(r-2 \ r-1) \vartriangleright (r-1 \  r)] \cdots]}g_{( k \ k+1)}^{-1}\\
    &=g_{( k \  k+1)}g_{( k+1 \  k+2)}g_{(k+2 \  k+3) \vartriangleright [ \cdots \vartriangleright [(r-2 \ r-1) \vartriangleright (r-1 \  r)] \cdots]}g_{( k+1 \ k+2)}^{-1}g_{( k \ k+1)}^{-1}\\
    &\ \ \vdots\\
    &= g_{( k \  k+1)}g_{( k+1 \  k+2)} \cdots g_{(r-2 \ r-1)}g_{(r-1 \  r)}g_{(r-2 \ r-1)}^{-1} \cdots g_{( k+1 \ k+2)}^{-1}g_{( k \ k+1)}^{-1}.
\end{align*}
Hence, the set of elements $\{g_{(i \  i+1)} \in As(\mathbb{P}_{n}) \ | \ i = 1,\dots,n-1\}$ is a generating set of $As(\mathbb{P}_{n})$ and thus, the set $\{g_{(i \  i+1)}\langle g_{(12)}^{2} \rangle \in \overline{G_{\mathbb{P}_{n}}} \ | \ i = 1,\dots,n-1\}$ is a generating set of $\overline{G_{\mathbb{P}_{n}}}$.
Now, note that for every $i=1,\dots,n-1$, we have that $(i \  i+1)\vartriangleright (i+1 \ i+2) = (i \  i+1)(i+1 \ i+2)(i \  i+1) = (i \ i+2)$ and $(i+1 \ i+2)\vartriangleright (i \  i+1) = (i+1 \ i+2)(i \  i+1)(i+1 \ i+2) = (i \ i+2)$. Therefore,
\begin{align*}
    (g_{(i \ i+1)}g_{(i+1 \ i+2)})^{3} &= g_{(i \  i+1)}g_{(i+1 \ i+2)}g_{(i \  i+1)}g_{(i+1 \ i+2)}g_{(i \  i+1)}g_{(i+1 \ i+2)}\\
    &= g_{(i \  i+1)\vartriangleright (i+1 \ i+2)}g_{(i+1 \ i+2)\vartriangleright (i \  i+1)}\\
    &= g_{(i \ i+2)}g_{(i \ i+2)} = g_{(i \ i+2)}^{2}.
\end{align*}
It follows that $(g_{(i \ i+1)}g_{(i+1 \ i+2)}\langle g_{(12)}^{2} \rangle)^{3}=  g_{(i \ i+2)}^{2} \langle g_{(12)}^{2} \rangle = \langle g_{(12)}^{2} \rangle = 1_{\overline{G_{\mathbb{P}_{n}}}}$. \\
Let $i,j \in \{1,\dots,n-1\}$ such that $|j-i|>1$. We want to prove that $(g_{(i \ i+1)}g_{(j \ j+1)}\langle g_{(12)}^{2} \rangle)^{2} =1_{\overline{G_{\mathbb{P}_{n}}}} $. Since $|j-i|>1$ we have two cases:
\begin{itemize}
    \item If $j-i>1$, then $j>1+i$ then 
    $$(i \ \ i+1) \vartriangleright (j \ \ j+1) =(i \ \ i+1)(j \ \ j+1)(i \ \ i+1)= (j \ \ j+1),$$
    and therefore
    \begin{align*}
        (g_{(i \ i+1)}g_{(j \ j+1)}\langle g_{(12)}^{2} \rangle)^{2} &= (g_{(i \ i+1)}g_{(j \ j+1)}\langle g_{(12)}^{2} \rangle)(g_{(i \ i+1)}g_{(j \ j+1)}\langle g_{(12)}^{2} \rangle)\\
        &=(g_{(i \ i+1)}g_{(j \ j+1)}g_{(i \ i+1)})g_{(j \ j+1)}\langle g_{(12)}^{2} \rangle\\
        &= g_{(i \ i+1) \vartriangleright (j \ j+1)}g_{(j \ j+1)}\langle g_{(12)}^{2} \rangle\\
        &=g_{(j \ j+1)}g_{(j \ j+1)}\langle g_{(12)}^{2} \rangle\\
        &= \langle g_{(12)}^{2} \rangle = 1_{\overline{G_{\mathbb{P}_{n}}}}.
 \end{align*}
 \item If $j-i<-1$, then $j<i-1$ then 
 $$(i \ \ i+1) \vartriangleright (j \ \ j+1) =(i \ \ i+1)(j \ \ j+1)(i \ \ i+1)= (j \ \ j+1),$$ 
 and therefore
 \begin{align*}
        (g_{(i \ i+1)}g_{(j \ j+1)}\langle g_{(12)}^{2} \rangle)^{2} 
        &= \langle g_{(12)}^{2} \rangle = 1_{\overline{G_{\mathbb{P}_{n}}}}.
 \end{align*}
\end{itemize}
Then, the group $\overline{G_{\mathbb{P}_{n}}}$ satisfies all the relations of the presentation of $\mathbb{S}_{n}$. Thus, there exists a group epimorphism $\phi: \mathbb{S}_{n} \longrightarrow \overline{G_{\mathbb{P}_{n}}}$, which implies that $|\overline{G_{\mathbb{P}_{n}}}| \leq n!$. Besides, we have that the homomorphism $\bar{\psi}: \overline{G_{\mathbb{P}_{n}}} \longrightarrow \mathbb{S}_{n}$ is surjective, then $|\mathbb{S}_{n}| \leq |\overline{G_{\mathbb{P}_{n}}}|$. Because of the fact that $|\overline{G_{\mathbb{P}_{n}}}| \leq n!$, we have that $|\overline{G_{\mathbb{P}_{n}}}| = n!$. Thus, $\bar{\psi}$ must be bijective, i.e, the map
\begin{align*}
    \bar{\psi}: \overline{G_{\mathbb{P}_{n}}} &\longrightarrow \mathbb{S}_{n}\\
    g_{(i \ j)}\langle g_{(12)}^{2} &\rangle \longmapsto (i \ j),
\end{align*}

 is a group isomorphism.
\end{proof}

\section{Finitely stable racks}\label{Section 1.3}

Elhamdadi and Moutuou in \cite{Elhamdadi} define a new type of racks, called finitely stable, in an attempt to capture the notion of identity and center in the category of racks and quandles. In this section we give a short review about this type of racks.

\begin{Notation}
If ${u_{1}, u_{2},\dots, u_{n}}$ are elements in a rack $X$, then for an element $x \in X$ we write
	\begin{center}
	    $(u_{i})_{i=1} ^n \vartriangleright x = u_{n} \vartriangleright (\cdots (u_{3} \vartriangleright (u_{2} \vartriangleright (u_{1} \vartriangleright x)))\cdots)  $.\\
	 \end{center}
\end{Notation}

\begin{defi}
A \textit{stabilizer} $u$ in a rack $X$ is an element such that
\begin{center}
$u \vartriangleright x = x$  for all $x \in X$.
\end{center}

\end{defi}

Note that, a stabilizer $u$ in a rack $X$ satisfies that $L_{u}(x)=x$ for all $x \in X$, then $L_{u}$ is the identity function in the group $Inn(X)$.

\begin{eg}
Let G be a group and $u$ a stabilizer in $Conj(G)$, we have $g = u \vartriangleright g =ugu^{-1}.$ Hence, $ug=gu$, for all $g \in G$. Thus $u$ belongs to the center $Z(G)$ of G . Then, the stabilizers of $Conj(G)$ are the elements of $Z(G)$.
\end{eg}

The previous example shows us that the definition of stabilizer allows us to capture the notion of the center of a rack. Elhamdadi and Moutuou in \cite{Elhamdadi} take the property that a rack has a stabilizer and they weaken it by the following definition:

\begin{defi}
Let $X$ be a rack 
	    \begin{itemize}
	        \item [1.] A \textbf{stabilizing family of order $n$} for $X$ is a finite subset $\{u_{1}, \dots , u_{n} \}$ of $X$ such that,  $(u_{i})_{i=1} ^n \vartriangleright x = x, \   \forall x \in X.$ In other words, $$L_{u_{n}}L_{u_{n-1}}\cdots L_{u_{1}} = id.$$
	       \item[2.] The \textbf{$n$-center} X, denoted by $S^{n}(X)$, is the collection of all stabilizing families of order $n$ for $X$.
	      \item[3.] The collection  $S(X):= \underset{n \in \mathbb{N}}{\cup} S^{n}(X)$, of all stabilizing families for, $X$ is called the \textbf{center of $X$}.
        \item[4.]  A rack $X$ is said to be \textbf{\textit{finitely stable}} if $S(X) \neq \emptyset$. Further, if $S^{n}(X) \neq \emptyset$ for some $n \in \mathbb{N}$, then the rack is said to be \textbf{\textit{n- stable}}.
     	    \end{itemize}

\end{defi}

  Let $X$ be a finite rack of order $n$ and let $x \in X$, since $L_{x} \in Sym(X)$ then $L_{x}^{n!}= id$. Hence, the family $\{x\}_{i=1}^{n!}$  is a stabilizing family for $X$. So, $S(X) \neq \emptyset$. Thus, we have the proof of the following lemma.
  
  \begin{lemma} 
    
     Every finite rack is finitely stable. 
	  \end{lemma}
	   	 
The next theorem characterizes the stabilizing families of the conjugacy quandle $Conj(G)$ of a group $G$. The proof can be found in \cite{Elhamdadi}. 

\begin{theo}\label{Theo 1.4}
Let G be a group, then $\{u_{i}\}_{i=1} ^{n} \in S^{n}(Conj(G))$  if and only if $u_{n}u_{n-1}\cdots u_{1} \in Z(G)$
\end{theo}

\section{Representation of racks}

Let us begin this section with a short review of a rack representation, see \cite{Elhamdadi} for more details.

\begin{defi}
    A \textbf{\textit{representation}} of a rack $X$ is a vector space $V$ equipped with a map $\rho : X \longrightarrow Conj(GL(V))$, $
		  x\longmapsto \rho_{x},$ which is a rack homomorphism, i.e., $\rho_{x \vartriangleright y} = \rho_{x}\rho_{y}\rho_{x}^{-1}$ for all $x,y \in X$.
\end{defi}
It is not hard to prove that if $\rho : X \longrightarrow Conj(GL(V))$ is a rack representation, then $\bullet: X\times V \rightarrow V$, where $x\bullet v=\rho_{x}(v)$ defines a rack action in the sense of \cite[Definition 5.1]{Elhamdadi}.
\begin{eg}
    Let $G$ be a group, then every representation of $G$ defines a representation of the quandle $Conj(G)$. Indeed, let $\rho: G \longrightarrow GL(V)$ be  a group representation, that means, $\rho$ is a group homomorphism. Let $g,h \in G$, note that,   
        $\rho_{g \vartriangleright h} = \rho_{ghg^{-1}} = \rho_{g}\rho_{h}\rho_{g}^{-1}.$  
    Therefore, the map $\rho: Conj(G) \longrightarrow Conj (GL(V))$ is also a rack homomorphism.
\end{eg}

As in representations of groups, we can define the regular representation of a rack $X$.
\begin{lemma}
    Let $X$ be a finite rack and $\mathbb{C}X$ the complex vector space, seen as the set of formal sums
	    \begin{center}
	        $\mathbb{C}X=\{ f= \underset{x \in X}{\sum} f_{x} x \ | \ f_{x} \in \mathbb{C}\ \ and\  \ f_{x} = 0 \ \text{except for finitely many x's}\}.$
	    \end{center}
 Then, the map $\lambda : X \longrightarrow Conj(GL(\mathbb{C}X))$, defined by
\begin{align*}
    \lambda_{t} (f) = \lambda_{t} (\underset{x \in X}{\sum} f_{x} x)  := \underset{x \in X}{\sum} f_{x}(t \vartriangleright x) =  \underset{u \in X}{\sum} f_{L^{-1}_{t}} u= f \circ L_{t}^{-1},
\end{align*}
is a representation of $X$.
\end{lemma} \label{lemma 2.2}
\begin{proof} Let $t \in X$ and $f \in \c X$, note that $\lambda_{t}^{-1}(f)=f(L_{t})$, indeed, we have $\lambda_{t}(\lambda_{t}^{-1}(f))=\lambda_{t}(f(L_{t}))= f(L_{t}L_{t}^{-1})=f$ and $\lambda_{t}^{-1}(\lambda_{t}(f))=\lambda_{t}(f(L_{t}^{-1}))= f(L_{t}^{-1}L_{t})=f$. Now, let $x \in X$, then  for every $z \in X$ we have,
\begin{align*}
    \lambda_{t \vartriangleright x}(f)(z) &= f(L_{t \vartriangleright x}^{-1}(z))=f((L_{t}L_{x}L_{t}^{-1})^{-1}(z))\\
    &=(fL_{t}L_{x}^{-1})(L_{t}^{-1}(z))=\lambda_{t}((fL_{t})(L_{x}^{-1}(z)))=\lambda_{t}\lambda_{x}(f(L_{t}(z)))\\
    &=\lambda_{t}\lambda_{x}\lambda_{t}^{-1}(f)(z).
\end{align*}
Therefore, $\lambda_{t\vartriangleright x} =\lambda_{t}\lambda_{x}\lambda_{t}^{-1} $.
\end{proof}

\begin{defi}
 Let $X$ be a rack, the representation described in Lemma \ref{lemma 2.2} is called the \textit{\textbf{regular representation}} of $X$. 
\end{defi}

\begin{eg}\label{example 2.5}

    Consider the Takasaki quandle $Q= (\mathbb{Z}_{3}, \vartriangleright)$, which is the additive group $\mathbb{Z}_{3}$ with the operation $x \vartriangleright y = 2x -y$. The table of this quandle is:

\begin{center}
\begin{tabular}{c|c|c|c}
$\vartriangleright$ & \textbf{0} & \textbf{1}  & \textbf{2}\\ \hline
$\textbf{0}$ & 0 & 2 & 1 \\ \hline
$\textbf{1}$ & 2 & 1 & 0 \\ \hline
$\textbf{2}$ & 1 & 0 & 2 
\end{tabular}
\end{center}
\end{eg}
The set of functions $\{\delta_{0},\delta_{1},\delta_{2}\}$ is a basis for the quandle ring $\c Q$. The regular representation $\lambda: Q \longrightarrow Conj(GL(\c Q))$ is defined by the following equalities,
\begin{align*}
    \lambda_{0}(\delta_{0}(x))= \delta_{0}(L_{0}^{-1}(x))=\delta_{0}(L_{0}(x)) = \delta_{0}(x),\\
    \lambda_{0}(\delta_{1}(x))= \delta_{1}(L_{0}^{-1}(x))=\delta_{1}(L_{0}(x)) = \delta_{2}(x),\\
    \lambda_{0}(\delta_{2}(x))= \delta_{2}(L_{0}^{-1}(x))=\delta_{2}(L_{0}(x)) = \delta_{1}(x),\\
    \\
    \lambda_{1}(\delta_{0}(x))= \delta_{0}(L_{1}^{-1}(x))=\delta_{0}(L_{1}(x)) = \delta_{2}(x),\\
    \lambda_{1}(\delta_{1}(x))= \delta_{1}(L_{1}^{-1}(x))=\delta_{1}(L_{1}(x)) = \delta_{1}(x),\\
    \lambda_{1}(\delta_{2}(x))= \delta_{2}(L_{1}^{-1}(x))=\delta_{2}(L_{1}(x)) = \delta_{0}(x),\\
    \\
    \lambda_{2}(\delta_{0}(x))= \delta_{0}(L_{2}^{-1}(x))=\delta_{0}(L_{2}(x)) = \delta_{1}(x),\\
    \lambda_{2}(\delta_{1}(x))= \delta_{1}(L_{2}^{-1}(x))=\delta_{1}(L_{2}(x)) = \delta_{0}(x),\\
    \lambda_{2}(\delta_{2}(x))= \delta_{2}(L_{
    2}^{-1}(x))=\delta_{2}(L_{2}(x)) = \delta_{2}(x).
\end{align*}
Therefore, with the basis $\{\delta_{0},\delta_{1},\delta_{2}\}$, we can describe the regular representation $\lambda: Q \longrightarrow Conj(GL(3,\c))$  in matrix form as follows:
\begin{align*}
       \lambda_{0}=  \left[
             \begin{matrix}
              1 & 0 & 0 \\
              0 & 0 & 1\\
              0 & 1 & 0
             \end{matrix}
                   \right], \ 
        \lambda_{1}=  \left[
             \begin{matrix}
              0 & 0 & 1 \\
              0 & 1 & 0\\
              1 & 0 & 0
             \end{matrix}
                   \right],
        \lambda_{2}=  \left[
             \begin{matrix}
              0 & 1 & 0 \\
              1 & 0 & 0\\
              0 & 0 & 1
             \end{matrix}
                   \right].
\end{align*}

\begin{defi}
    Let $(\rho, V)$ and $(\phi,W)$ be two representations of a rack X. A linear map $T: V \longrightarrow  W$ is called $X$- \textbf{\textit{linear}} if for all $x \in X$ the following diagram commutes

\centerline{\xymatrix{ V \ar[r]^{\rho_{x}} \ar[d]_T & V \ar[d]^{T} \\ W \ar[r]_{\phi_{x}} & W, }}

that means, $\phi_{x} T = T \rho_{x}$, for all $x \in X$.
The representations $\phi$ and $\rho$ are said to be \textbf{\textit{ equivalent}} if $T$ is an isomorphism.
\end{defi}

As in the case of groups, the previous definition defines an equivalence relation of the set of rack representations.  We use the notation $\phi \sim \rho$ for two equivalent representations.

\begin{defi}
 Let $\rho: X \longrightarrow  Conj(GL(V))$ be a representation of a rack $X$ and $W \subset V$ a subspace of $V$, such that $\rho_{x}(W) \subset W$ for all $x \in X$, then $W$ is called a \textbf{\textit{ subrepresentation}}. A representation $V$ of $X$ is said to be \textbf{\textit{ irreducible}} if the only subrepresentations are $W=\{0\}$ and $W=V$.
\end{defi}
 With the last definition we have completed our review about rack representation theory. Let us illustrate some definitions with more examples.

\begin{eg} \label{Example 1.11}
The permutation quandle $\mathbb{P}_{3}$, see Example \ref{Example 1.8}, is isomorphic to the Takasaki quandle $\mathbb{Z}_{3}$, where the isomorphism is the function $\phi: \z_{3} \longrightarrow \mathbb{P}_{3}$ defined by $\phi(0):= (2 \ 3)$, $\phi(1):= (1 \ 3)$ and $\phi(2):= (1 \ 2).$ Thus, the regular representation of this quandle is $\lambda: \mathbb{P}_{3} \longrightarrow Conj(GL(3,\c))$  defined by
\begin{align*}
       \lambda_{(2 \ 3)}=  \left[
             \begin{matrix}
              1 & 0 & 0 \\
              0 & 0 & 1\\
              0 & 1 & 0
             \end{matrix}
                   \right], \ 
        \lambda_{(1 \ 3)}=  \left[
             \begin{matrix}
              0 & 0 & 1 \\
              0 & 1 & 0\\
              1 & 0 & 0
             \end{matrix}
                   \right],
        \lambda_{(1 \ 2)}=  \left[
             \begin{matrix}
              0 & 1 & 0 \\
              1 & 0 & 0\\
              0 & 0 & 1
             \end{matrix}
                   \right].
\end{align*}

\end{eg}

\begin{eg}\label{Example 2.7}
  Let us consider $X=\{1,2,3\}$ with the operation $i \vartriangleright j = \sigma(j)$ for all $i,j \in X$, where $\sigma = (1 \ 2\ 3) \in \mathbb{S}_{3}$. $(X , \vartriangleright)$ is a rack (see Example \ref{Example 1.12})  with the table  
\begin{center}
\begin{tabular}{c|c|c|c}
$\vartriangleright$ & \textbf{1} & \textbf{2}  & \textbf{3}\\ \hline
$\textbf{1}$ & 2 & 3 & 1 \\ \hline
$\textbf{2}$ & 2 & 3 & 1 \\ \hline
$\textbf{3}$  & 2 & 3 & 1 
\end{tabular}
\end{center}
\end{eg}

Note that $L_{i}^{3} =id$ for all $i \in X$, then $L_{i}^{2} = L_{i}^{-1}$ for all $i \in X$. The regular representation $\lambda: X \longrightarrow Conj(GL(\c X))$ is defined by
\begin{align*}
    \lambda_{1}(\delta_{1}(x))= \delta_{1}(L_{1}^{-1}(x))=\delta_{1}(L_{1}^{2}(x)) = \delta_{2}(x),\\
    \lambda_{1}(\delta_{2}(x))= \delta_{2}(L_{1}^{-1}(x))=\delta_{2}(L_{1}^{2}(x)) = \delta_{3}(x),\\
    \lambda_{1}(\delta_{3}(x))= \delta_{3}(L_{1}^{-1}(x))=\delta_{3}(L_{1}^{2}(x)) = \delta_{1}(x).
\end{align*}

Since $L_{1}=L_{2}=L_{3}$ then $\lambda_{1} = \lambda_{2} = \lambda_{3}$. Then we have that

\begin{align*}
       \lambda_{1} = \lambda_{2}=\lambda_{3}=  \left[
             \begin{matrix}
              0 & 0 & 1 \\
              1 & 0 & 0\\
              0 & 1 & 0
             \end{matrix}
                   \right].
\end{align*}

Let $\varphi: G\rightarrow GL(V)$ be a group representation. Since, 
\[
\varphi(xyx^{-1})=\varphi(x)\varphi(y)\varphi(x)^{-1},
\]
for every $x,y\in G$, we have the following  equality $\varphi(x\triangleright y)=\varphi(x)\triangleright\varphi(y)$. Thereby $\varphi \in Hom_{rack}(Conj(G)),Conj(GL(V))$ and then,
\[
Hom_{gr}(G,GL(V)) \subset Hom_{rack}(Conj(G), Conj(GL(V))).
\]
The other contention is not always true. For example, for  every $k \not\in \{0,1\}$, the map $\varphi_k: Conj(G)\rightarrow Conj(\mathbb{C}^{*})$, given by $\varphi(g)=k$, for every $g\in G$, is a rack homomorphism but not a group homomorphism. Note that the rack representations $\varphi_k$ are irreducibles, because they are one dimensional and $\varphi_k \sim \varphi_r$ if and only if $k=r$. Thus, we have that the number of irreducible representations (up to equivalence) for the finite rack $Conj(G)$ is infinite, which doesn't happen in group theory, where the number of irreducible representations (up to equivalence) of finite groups is finite.  

\begin{theo}\label{bijtheo} 
With the above notation, 
    \[
    Hom_{gr}(As(X),GL(V)) \cong Hom_{rack}(X, Conj(GL(V))).
    \]
    Such isomorphism induces a biunivocal correspondence between the equivalence classes of representations. 
\end{theo}
\begin{proof}
    Theorem \ref{Theorem 1.2} guarantees the existence of an isomorphism $\Psi$ between $Hom_{gr}(As(X),GL(V))$ and $ Hom_{rack}(X, Conj(GL(V)))$. The proof that $\Psi$ induces a biunivocal correspondence between the equivalence classes of representations comes from a direct calculation. Now, let $\varphi: X \rightarrow Conj(GL(V))$ be a reducible representation, then there exists a non trivial $X$- invariant subspace $W$ of $V$. Thus, $\phi: X \rightarrow Conj(GL(W))$, defined by $\phi(x)=\varphi(x)$, for every $x\in X$,  is a sub-representation of $\varphi$. From Theorem \ref{Theorem 1.2}, we can define $\widehat{\phi}: As(X)\rightarrow GL(W)$, which is a non trivial sub-representation of $\widehat{\varphi}$, therefore, the induced group representation $\widehat{\varphi}$ is reducible. The reciprocal is also true, this is because any reducible group representation $\psi: As(X)\rightarrow GL(V) $ induces a reducible non-trivial rack representation of $X$.    
\end{proof}
Thus, from the previous theorem, if for a group $G$,  $As(Conj(G))=G$, then 
\[
Hom_{gr}(G,GL(V)) = Hom_{rack}(Conj(G), Conj(GL(V))).
\]
For finite groups the equality $As(Conj(G))=G$ is not true, this is because $As(Conj(G))$ is always infinite, see Proposition \ref{asinfi}.  

Besides, note that, if a rack $X$ is the trivial rack, then the associate group $As(X)$ is abelian. Thus, for every abelian group, $As(Conj(G))$ is an abelian group. Moreover, if $G$ is an abelian group of order $m$, then $As(Conj(G))\cong \mathbb{Z}^{m}$. 

The proof of the following corollary is a direct application of Theorem \ref{bijtheo}, so we omit it.
\begin{cor}
    The irreducible representations of trivial racks are one dimensional. 
\end{cor}
Another important result is the following.
\begin{prop}
Any rack $X$ has infinite many equivalence classes of irreducible representations.
\end{prop}
\begin{proof}
    The proof comes from the fact that if $k\in \mathbb{C}^{*}$, then $\varphi_k: X\rightarrow Conj(\mathbb{C}^{*})$, where $\varphi_k(x)=k$, for all $x\in X$, is an irreducible representation. Moreover, if $k\neq r$, then $\varphi_{k}$ is not equivalent to $\varphi_r$.
\end{proof}

\section{Strong representations}
Elhamdadi and Moutuou in \cite[Definition 9.1] {Elhamdadi} introduce the concept of \textit{strong representations} and they show some general properties of them. In this section, we take that definition and we present new and interesting results about this type of rack representations. 
\begin{defi}
 A representation $\rho: X \longrightarrow GL(V)$ of a rack $X$ is said to be \textbf{\textit{strong}} if every stabilizing family $\{x_{1},\dots,x_{n}\} \subset X$,  satisfies that,
 \[\rho_{x_{n}}\rho_{x_{n-1}}\cdots \rho_{x_{1}} = id. \] In other words, if $L_{x_{n}}L_{x_{n-1}}\cdots L_{x_{1}}=id$ implies that $\rho_{x_{n}}\rho_{x_{n-1}}\cdots \rho_{x_{1}} = id $.
\end{defi}

\begin{prop}\label{Proposition 2.1}
Let $X$ be a rack. The regular representation $\lambda: X \longrightarrow Conj(GL(\c X))$ is strong.
\end{prop}

\begin{proof}
    Let $\{u_{1},\dots,u_{n}\}$ be a stabilizing family of the rack $X$ and let $f \in \c X$. Note that
\begin{align*}
       (\rho_{u_{n}}\cdots \rho_{u_{2}}\rho_{u_{1}})(f) &= (\rho_{u_{n}}\cdots \rho_{u_{2}})(\rho_{u_{1}}(f))=(\rho_{u_{n}}\cdots \rho_{u_{2}})(fL_{u_{1}}^{-1})\\
    &=(\rho_{u_{n}}\cdots \rho_{u_{3}})[\rho_{u_{2}}(fL_{u_{1}}^{-1}))]=(\rho_{u_{n}}\cdots \rho_{u_{3}})(fL_{u_{1}}^{-1}L_{u_{2}}^{-1})\\
    &= fL_{u_{1}}^{-1}L_{u_{2}}^{-1} \cdots L_{u_{n}}^{-1}= f(L_{u_{n}} \cdots L_{u_{2}}L_{u_{1}})^{-1}.
\end{align*}
Since $\{u_{1},\dots,u_{n}\}$ is a stabilizing family, then $L_{u_{n}} \cdots L_{u_{1}} = id$, therefore      $(\rho_{u_{n}}\cdots \rho_{u_{2}}\rho_{u_{1}})(f) = f(L_{u_{n}} \cdots L_{u_{2}}L_{u_{1}})^{-1}=f$.
\end{proof}

We have found a connection between strong representations and representations of finite groups. To show this, let us introduce the following notation.

\begin{Notation}
    From now on, for a finite connected rack $X$, we write $n$ for the order of all permutations $L_{x}$ and $x_{0}$ for a fixed element of $X$.
\end{Notation}

In the next theorem, we prove that every irreducible strong representation of a finite connected rack $X$ induces an irreducible representation of the finite enveloping group $\overline{G_{X}}$. 

\begin{theo} \label{Theorem 2.4}
    Let $X$ be a finite-connected rack and let $(V,\rho)$ be a strong representation of $X$. Then $\rho$ induces a  representation $\bar{\rho}: \overline{G_{X}} \longrightarrow GL(V)$ of the finite enveloping group $\overline{G_{X}}$ such that $\bar{\rho}_{g_{x} \langle g_{x_{0}}^{n}\rangle} := \rho_{x}$ for all $x \in X$. Furthermore, if $\rho$ is an irreducible rack representation, then $\bar{\rho}$ is an irreducible group representation. 
\end{theo}
\begin{proof} Since $\rho: X \longrightarrow Conj(GL(V))$ is a rack homomorphism then, from Theorem \ref{Theorem 1.2}, $\rho$ induces a  group homomorphism $\hat{\rho}: As(X) \longrightarrow GL(V)$ such that $\hat{\rho}_{g_{x}} = \rho_{x}$ for all $x \in X$. Now,  fix $x_{0} \in X$, note that $\hat{\rho}_{g_{x_{0}}^{n}} = \hat{\rho}^{n}_{g_{x_{0}}} = \rho^{n}_{x_{0}}$. Since $n$ is the order of all permutations $L_{x}$, then $L^{n}_{x_{0}}=id.$ Due to the fact that the representation $\rho$ is strong, we have $\hat{\rho}_{x_{0}^{n}}= \rho^{n}_{x_{0}} =  id$. Therefore, $\langle g_{x_{0}}^{n} \rangle \subset ker(\hat{\rho})$. This implies that there exists a group homomorphism $\bar{\rho}: As(X)/\langle g_{x_{0}}^{n} \rangle \longrightarrow GL(V)$ such that $\bar{\rho}_{g_{x} \langle g_{x_{0}}^{n} \rangle} = \rho_{x}$ for all $x \in X$.\\
Now, suppose that $\rho$ is irreducible and suppose that $W$ is a subspace of $V$ such that $\bar{\rho}_{h}(W) \subset W$ for all $h \in \overline{G_{X}}$. In particular, we have that $\bar{\rho}_{g_{x} \langle g_{x_{0}}^{n} \rangle}(W) = \rho_{x}(W) \subset W$ for all $x \in X$. Since $\rho$ is irreducible, then $W=\{0\}$ or $W=V$. Therefore, $\bar{\rho}$ is also irreducible.
\end{proof}

Now, we prove that the group representation induced by a strong representation of a finite connected rack is well-defined in the set of conjugacy classes.  

\begin{theo} \label{Theorem 2.5}
    Let $X$ be a finite-connected rack, and let $(V,\rho)$  and $(V,\phi)$  be strong representations of $X$ such that $\rho \sim \phi$. Then, the induced group representations $\bar{\rho}: \overline{G_{X}} \longrightarrow GL(V)$  and $\bar{\phi}: \overline{G_{X}} \longrightarrow GL(V')$ are also equivalent.
\end{theo}
\begin{proof} Since  $\rho \sim \phi$ then, there exists an  isomorphism $T: V' \longrightarrow V$ such that $\rho_{x}T=T\phi_{x}$ for all $x \in X$. Let $h \in \overline{G_{X}}$, from the proof of Theorem \ref{Theorem 1.5} the element $h$ is the form $h = g_{x_{1}}^{e_{1}} \cdots g_{x_{m}}^{e_{m}} \langle g_{x_{0}}^{n} \rangle$ where $m \leq |\overline{G_{X}}|$, $x_{i} \in X$ for all $i \in \{1,\dots,m\}$ and $e_{i} \in \{0,1,\dots,n-1\}$. Then,
\begin{align*}
     \bar{\rho}_{h} T&= \bar{\rho}_{g_{x_{1}}^{e_{1}} \cdots g_{x_{m}}^{e_{m}} \langle g_{x_{0}}^{n} \rangle}T =\bar{\rho}^{e_{1}}_{g_{x_{1}} \langle g_{x_{0}}^{n} \rangle} \cdots \bar{\rho}^{e_{m}}_{g_{x_{m}}\langle g_{x_{0}}^{n} \rangle}T= \rho^{e_{1}}_{x_{1}} \cdots \rho^{e_{m}}_{x_{m}}T.
\end{align*}
Note that, for every $i \in \{1,\dots,m\}$ we have that
\begin{align*}
    \rho^{e_{i}}_{x_{i}} T &=  \rho^{e_{i} -1 }_{x_{i}} \rho_{x_{i}}T= \rho^{e_{i} -1 }_{x_{i}} T\phi_{x_{i}}= \rho^{e_{i} -2 }_{x_{i}}\rho_{x_{i}}T \phi_{x_{i}} = \rho^{e_{i} -2 }_{x_{i}}T \phi^{2}_{x_{i}}=  \cdots= T \phi^{e_{i}}_{x_{i}}.
\end{align*}
Therefore,  $\rho^{e_{1}}_{x_{1}} \cdots \rho^{e_{m}}_{x_{m}}T =\rho^{e_{1}}_{x_{1}} \cdots T\phi^{e_{m}}_{x_{m}} = T \phi^{e_{1}}_{x_{1}} \cdots \phi^{e_{m}}_{x_{m}}$. Hence, 
\begin{align*}
    \bar{\rho}_{h} T&=\rho^{e_{1}}_{x_{1}} \cdots \rho^{e_{m}}_{x_{m}}T= T \phi^{e_{1}}_{x_{1}} \cdots \phi^{e_{m}}_{x_{m}}=  T \bar{\phi}^{e_{1}}_{g_{x_{1}} \langle g_{x_{0}}^{n} \rangle} \cdots \bar{\phi}^{e_{m}}_{g_{x_{m}}\langle g_{x_{0}}^{n} \rangle}\\
    &= T\bar{\phi}_{g_{x_{1}}^{e_{1}}\langle g_{x_{0}}^{n} \rangle} \cdots \bar{\phi}_{g_{x_{m}}^{e_{m}}\langle g_{x_{0}}^{n} \rangle}= T \bar{\phi}_{g_{x_{1}}^{e_{1}} \cdots g_{x_{m}}^{e_{m}} \langle g_{x_{0}}^{n} \rangle}= T \bar{\phi}_{h}.
\end{align*}
Hence, the result follows.
\end{proof}

\begin{cor} \label{corollary 2.2}
Let $X$ be a finite connected rack, then the number of irreducible strong complex representations of $X$ (up to equivalence) is less than or equal to the number of conjugacy classes of the finite enveloping group $\overline{G_{X}}$.
\end{cor}
\begin{proof} Suppose that the number of conjugacy classes of $\overline{G_{X}}$ is $k \in \n$ . From representation theory of finite groups we have that the number of irreducible complex representations (up to equivalence) of the finite enveloping group $\overline{G_{X}}$ is equal to the number of conjugacy classes of $\overline{G_{X}}$. Now, reasoning by contradiction, suppose that the number of irreducible strong representations of $X$ is  $m > k$. Let $\rho_{1},\dots,\rho_{m}$ be the distinct representatives of irreducible strong representations (up to equivalence) of $X$. Then, from Theorem \ref{Theorem 2.4} every representative $\rho_{i}$ induces an irreducible representation $\bar{\rho_{i}}$ of the group $\overline{G_{X}}$. Since $\rho_{i} \not\sim \rho_{j}$ for all $i \neq j \in \{1,2,\dots,m\}$, then from Theorem \ref{Theorem 2.5} $\bar{\rho_{i}} \not\sim \bar{\rho_{j}}$ for all $i \neq j \in \{1,2,\dots,m\}$. Therefore, $\overline{G_{X}}$ would have $m >k$ irreducible representations (up to equivalence), which is a contradiction. Thus $m \leq k$.
\end{proof}

\begin{cor}
Let $X$ be a finite connected rack whose finite enveloping group $\overline{G_{X}}$ is abelian. For every irreducible strong representation $\rho: X \longrightarrow Conj(GL(V))$ of the rack $X$, we have that the vector space $V$ is one dimensional.
\end{cor}
\begin{proof} Let $\rho: X \longrightarrow Conj(GL(V))$ be an irreducible strong representation of $X$. From Theorem \ref{Theorem 2.4}, $\rho$ induces a irreducible representation $\bar{\rho}:  \overline{G_{X}} \longrightarrow GL(V)$ such that $\bar{\rho}_{g_{x} \langle g_{{x_{0}}}^{n} \rangle}$ = $\rho_{x}$ for all $x \in X$. Since $\overline{G_{X}}$ is an abelian group and $\bar{\rho}$ is irreducible, then $V$ must be of dimension one.
\end{proof}
Now, we consider the question if any group representation of $\overline{G_{X}}$, when $X$ is a finite-connected rack, can be lifted to a strong-representation of $X$. A partial answer is given in the next theorem.  

\begin{theo}\label{Theorem 2.6}
  Let $X$ be a finite connected rack. Let $\bar{\rho}: \overline{G_{X}} \longrightarrow GL(V)$ a representation of the finite enveloping group of $X$. Consider the function $\rho: X \longrightarrow Conj(GL(V))$ defined by $\rho_{x}:= \bar{\rho}_{g_{x} \langle g_{x_{0}}^{n} \rangle}$ for all $x \in X$. Then $\rho$ is a representation of the rack $X$. Furthermore, if $\bar{\rho}$ is irreducible, $\rho$ is too.
\end{theo}

\begin{proof} Let us see that $\rho$ is well defined. Let $x,y \in X$ such that $x=y$. Therefore, $g_{x} \langle g_{x_{0}}^{n} \rangle = g_{y} \langle g_{x_{0}}^{n} \rangle $. It follows that $\rho_{x} = \bar{\rho}_{g_{x} \langle g_{x_{0}}^{n} \rangle} = \bar{\rho}_{g_{y} \langle g_{x_{0}}^{n} \rangle} = \rho_{y}$, then $\rho$ is well defined. From the definition of  $\overline{G_{X}}$,  we have the relation $ g_{x \vartriangleright y} \langle g_{x_{0}}^{n} \rangle = g_{x}g_{y}g_{x}^{-1} \langle g_{x_{0}}^{n} \rangle$, for all $x,y \in X$. Then,
\begin{align*}
    \rho_{x \vartriangleright y} = \bar{\rho}_{g_{x \vartriangleright y} \langle g_{x_{0}}^{n} \rangle} = \bar{\rho}_{g_{x}g_{y}g_{x}^{-1}\langle g_{x_{0}}^{n} \rangle} = \bar{\rho}_{g_{x}\langle g_{x_{0}}^{n} \rangle}\bar{\rho}_{g_{y} \langle g_{x_{0}}^{n} \rangle} \bar{\rho}^{-1}_{g_{x} \langle g_{x_{0}}^{n} \rangle} = \rho_{x}\rho_{y}\rho^{-1}_{x}.
\end{align*}
Therefore, $\rho$ is a representation of the rack $X$. Suppose that $\bar{\rho}$ is irreducible. Let $W$ be a subspace of $V$ such that $\rho_{x}(W) \subset W$ for all $x \in X$. First, we claim that $\rho^{k}_{x}(W) \subset W$ for all $k \in \n$ and all $x \in X$. Indeed, suppose that for every $x \in X$ we have $\rho^{t}_{x}(W) \subset W$ for some $t \in \n$. Note that $\rho_{x}^{t+1}(W)= \rho_{x}\rho^{t}_{x}(W) = \rho_{x}[\rho^{t}_{x}(W)]$, since
$\rho^{t}_{x}(W) \subset W$,  then  we have $\rho_{x}^{t+1}(W)= \rho_{x}[\rho^{t}_{x}(W)] \subset W$, so the result follows by induction.\\
Let $h \in \overline{G_{X}}$, from the proof of Theorem \ref{Theorem 1.5}, the element $h$ is the form $h = g_{x_{1}}^{e_{1}} \cdots g_{x_{m}}^{e_{m}} \langle g_{x_{0}}^{n} \rangle$ where $m \leq |\overline{G_{X}}|$, $x_{i} \in X$ for all $i \in \{1,\dots,m\}$ and $e_{i} \in \{0,1,\dots,n-1\}$. Then,
\begin{align*}
     \bar{\rho}_{h}(W) &= \bar{\rho}_{g_{x_{1}}^{e_{1}} \cdots g_{x_{m}}^{e_{m}} \langle g_{x_{0}}^{n} \rangle}(W)= \bar{\rho}^{e_{1}}_{g_{x_{1}} \langle g_{x_{0}}^{n} \rangle} \cdots \bar{\rho}^{e_{m}}_{g_{x_{m}} \langle g_{x_{0}}^{n} \rangle}(W)= \rho^{e_{1}}_{x_{1}} \cdots \rho^{e_{m}}_{x_{m}}(W).
\end{align*}

Since $\rho^{e_{1}}_{x_{i}}(W)  \subset W$ and $\rho^{e_{2}}_{x_{i}}(W)  \subset W$ then, $\rho^{e_{1}}_{x_{1}}\rho^{e_{2}}_{x_{2}}(W) = \rho^{e_{1}}_{x_{1}}[\rho^{e_{2}}_{x_{2}}(W)] \subset W.$

Suppose that $\rho^{e_{1}}_{x_{1}} \cdots \rho^{e_{i}}_{x_{i}}(W) \subset W$ for some $i \in \{2,3,..,m-1\}$. Since $\rho^{e_{i+1}}_{x_{i+1}}(W)  \subset W$ it follows that $ \rho^{e_{1}}_{x_{1}} \cdots \rho^{e_{i}}_{x_{i}}\rho^{e_{i+1}}_{x_{i+1}}(W) = (\rho^{e_{1}}_{x_{1}} \cdots \rho^{e_{i}}_{x_{i}})[\rho^{e_{i+1}}_{x_{i+1}}(W)] \subset W$. Thus, from induction over $i$ we have $\rho^{e_{1}}_{x_{1}} \cdots \rho^{e_{m}}_{x_{m}}(W) \subset W$. Therefore, $\bar{\rho}_{h}(W)  \subset W$ for all $h \in \overline{G_{X}}$. Since $\bar{\rho}$ is irreducible, then $W=\{0\}$ or $W=V$. Hence, $\rho$ is an irreducible representation of $X$.
\end{proof}

We have a result similar to Theorem \ref{Theorem 2.5}.
  
\begin{theo} \label{Theorem 2.7}
  Let $X$ be a finite-connected rack. Let $\bar{\rho}: \overline{G_{X}} \longrightarrow GL(V)$  and $\bar{\phi}: \overline{G_{X}} \longrightarrow GL(V')$  be representations of the group $\overline{G_{X}}$ such that $\bar{\rho} \sim \bar{\phi}$. Then, the rack representations $\rho: X \longrightarrow Conj(GL(V))$  and, $\phi: X \longrightarrow Conj(GL(V'))$ defined as in the previous theorem, are also equivalent.
\end{theo}

\begin{proof} Since $\bar{\rho} \sim \bar{\phi}$ then there exists an isomorphism $T: V' \longrightarrow V$ such that $\bar{\rho}_{h}T = T \bar{\phi}_{h} $ for all $h \in \overline{G_{X}}$. In particular, $\bar{\rho}_{g_{x} \langle g_{x_{0}}^{n} \rangle }T = T \bar{\phi}_{g_{x} \langle g_{x_{0}}^{n} \rangle} $ for all $x \in X$. Thus, $\rho_{x} T = \bar{\rho}_{g_{x} \langle g_{x_{0}}^{n} \rangle }T = T \bar{\phi}_{g_{x} \langle g_{x_{0}}^{n} \rangle} = T \phi_{x}$ for all $x \in X$.Therefore, $\rho \sim \phi$.
 \end{proof}

From Theorem \ref{Theorem 2.6}, we can obtain a representation of a finite connected rack $X$ from its finite enveloping group $\overline{G_{X}}$. This representation may not necessarily be strong, but under certain conditions we can ensure this property.

\begin{theo}\label{Theorem 2.8}
  Let $X$ be a finite connected rack and $\bar{\rho}: \overline{G_{X}} \longrightarrow GL(V)$ be a representation of the finite enveloping group $\overline{G_{X}}$. If $\overline{G_{X}}$ has trivial center, that is, $Z(\overline{G_{X}})= \{1\}$ then the rack representation $\rho: X \longrightarrow Conj(GL(V))$ defined by $\rho_{x} := \bar{\rho}_{g_{x} \langle g_{x_{0}}^{n} \rangle}$, is strong.
\end{theo}

\begin{proof} Let $\{x_{1},\dots,x_{k}\}$ be a stabilizing family of the rack $X$. That means, $x_{k} \vartriangleright ( x_{n-1} \vartriangleright (\cdots(x_{1} \vartriangleright x) \cdots)) = x $ for all $x \in X$. Since we have the relation $g_{x \vartriangleright y} = g_{x}g_{y}g_{x}^{-1}$, for all $x,y \in X$, in the associated group $As(X)$ then for every $x \in X$ we have
\begin{align*}
     g_{x} &=g_{x_{k} \vartriangleright ( x_{k-1} \vartriangleright (\cdots(x_{1} \vartriangleright x) \cdots))} = g_{x_{k}}g_{x_{k-1} \vartriangleright (\cdots(x_{1} \vartriangleright x) \cdots)} g_{x_{k}}^{-1}\\
    &=g_{x_{k}}g_{x_{k-1}}g_{x_{k-2} \vartriangleright (\cdots(x_{1} \vartriangleright x) \cdots)}g_{x_{k-1}}^{-1}g_{x_{k}}^{-1} \\
    & \vdots\\
    &= (g_{x_{k}} \cdots g_{x_{1}})g_{x}(g_{x_{1}}^{-1} \cdots g_{x_{k}}^{-1}).
\end{align*}
Therefore, $(g_{x_{k}} \cdots g_{x_{1}})g_{x} = g_{x}(g_{x_{k}} \cdots g_{x_{1}})$. Thus, the word $g_{x_{k}}g_{x_{k-1}} \cdots g_{x_{1}}$ belongs to the center of the group $As(X)$ and therefore $g_{x_{k}}g_{x_{k-1}} \cdots g_{x_{1}} \langle g_{x_{0}}^{n} \rangle \in Z(\overline{G_{X}})$. Since $Z(\overline{G_{X}})= \{1\}$, then  $g_{x_{k}}g_{x_{k-1}} \cdots g_{x_{1}} \langle g_{x_{0}}^{n} \rangle = 1$.\\
Note that,
\begin{center}
    $\rho_{x_{k}} \cdots \rho_{x_{1}} = \bar{\rho}_{g_{x_{k}}\langle x^{n}_{0}} \rangle \cdots \bar{\rho}_{g_{x_{1}}\langle g_{x_{0}}^{n} \rangle} = \bar{\rho}_{g_{x_{k}}g_{x_{k-1}} \cdots g_{x_{1}} \langle g_{x_{0}}^{n} \rangle} = \bar{\rho}_{1} = id$.
\end{center}
 Hence, the representation $\rho$ is strong.
\end{proof}
\begin{cor} \label{Collorary 2.3}
Let $X$ be a finite-connected rack. If $Z(\overline{G_{X}}) = \{1\}$, then the number of irreducible strong complex representations of $X$ (up to equivalence) is equal to the number of irreducible complex representations of the group $\overline{G_{X}}$.
\end{cor}

\begin{proof} From previous theorems, there exists a bijective correspondence between irreducible strong representations of the rack $X$ and the irreducible representations of the group $\overline{G_{X}}$.

 \end{proof}

\section{Examples and counterexamples}
  In this section, we give some illustrative examples of the results given in the previous one. We also present a counterexample to \cite[Theorem~9.11]{Elhamdadi}.

\begin{eg} \label{Example 2.10}
  Consider the permutation quandle $\mathbb{P}_{3}$. In Example \ref{Example 1.11} we found the regular representation of $\mathbb{P}_{3}$, which is $\lambda: \mathbb{P}_{3} \longrightarrow Conj(GL(3,\c))$ defined by 
  \begin{align*}
       \lambda_{(2 \ 3)}=  \left[
             \begin{matrix}
              1 & 0 & 0 \\
              0 & 0 & 1\\
              0 & 1 & 0
             \end{matrix}
                   \right], \ 
        \lambda_{(1 \ 3)}=  \left[
             \begin{matrix}
              0 & 0 & 1 \\
              0 & 1 & 0\\
              1 & 0 & 0
             \end{matrix}
                   \right],
        \lambda_{(1 \ 2)}=  \left[
             \begin{matrix}
              0 & 1 & 0 \\
              1 & 0 & 0\\
              0 & 0 & 1
             \end{matrix}
                   \right].
\end{align*}
Remark that $\overline{G_{\mathbb{P}_{3}}} \cong \mathbb{S}_{3}$ (see Proposition \ref{proposition 1.4}) where the isomorphism $\mu: \overline{G_{\mathbb{P}_{3}}} \longrightarrow \mathbb{S}_{3}$ is given by $\mu_{g_{(ij)} \langle g_{(12)}^{2} \rangle} = (i \ j)$ for all $(i \ j) \in \mathbb{P}_{3}$. From Proposition \ref{Proposition 2.1} the regular representation $\lambda$ is strong. Then it induces a representation $\bar{\lambda}: \mathbb{S}_{3} \longrightarrow GL(3,\c)$ of the finite enveloping group $\overline{G_{\mathbb{P}_{3}}} \cong \mathbb{S}_{3}$ defined by
\begin{align*}
       \bar{\lambda}_{(2 \ 3)}=  \left[
             \begin{matrix}
              1 & 0 & 0 \\
              0 & 0 & 1\\
              0 & 1 & 0
             \end{matrix}
                   \right], \ 
        \bar{\lambda}_{(1 \ 3)}=  \left[
             \begin{matrix}
              0 & 0 & 1 \\
              0 & 1 & 0\\
              1 & 0 & 0
             \end{matrix}
                   \right],
        \bar{\lambda}_{(1 \ 2)}=  \left[
             \begin{matrix}
              0 & 1 & 0 \\
              1 & 0 & 0\\
              0 & 0 & 1
             \end{matrix}
                   \right].
\end{align*}
Note that, the subspace $W= span\left\{ \begin{bmatrix}
1 \\
1 \\
1 \\
\end{bmatrix}\right\}
$ is invariant under the representation $\bar{\lambda}$. Thus, this representation is reducible and decomposable. From the representation theory of finite groups, we know that every representation can be written as an unique (up to equivalence)  direct sum of irreducible representations. Specifically, we have that $\bar{\lambda} \sim \bar{\phi} \oplus \bar{\psi}$, where  $\bar{\phi}: \mathbb{S}_{3} \longrightarrow \c^{*}$  and $\bar{\psi}: \mathbb{S}_{3} \longrightarrow GL(2,\c)$  are irreducible representations defined by  $\bar{\phi}(g)=1$ for all $g \in \mathbb{S}_{3}$ and
\begin{align*}
       \bar{\psi}_{(1 \ 2)}:= \left[
             \begin{matrix}
              -1 & -1 \\
              0 & 1 \\
              \end{matrix}
                   \right], \  \bar{\psi}_{(1 \ 3)}:= \left[
             \begin{matrix}
              1 & 0 \\
              -1 & -1 \\
              \end{matrix}
                   \right],\  \bar{\psi}_{(2 \ 3)}:= \left[
             \begin{matrix}
              0 & 1 \\
              1 & 0 \\
              \end{matrix}
                   \right].
\end{align*}
From Theorem \ref{Theorem 2.6}, the group representations $\bar{\phi}$ and $\bar{\psi}$ induce rack representations $\phi: \mathbb{P}_{3} \longrightarrow Conj(GL(\c^{*}))$ and $\psi: \mathbb{P}_{3} \longrightarrow Conj(GL(2,\c))$ defined by $\phi_{x} := 1$ for all $x \in \mathbb{P}_{3}$ and
\begin{align*}
       \psi_{(1 \ 2)}:= \left[
             \begin{matrix}
              -1 & -1 \\
              0 & 1 \\
              \end{matrix}
                   \right], \  \psi_{(1 \ 3)}:= \left[
             \begin{matrix}
              1 & 0 \\
              -1 & -1 \\
              \end{matrix}
                   \right],\  \psi_{(2 \ 3)}:= \left[
             \begin{matrix}
              0 & 1 \\
              1 & 0 \\
              \end{matrix}
                   \right].
\end{align*}
From Theorem \ref{Theorem 2.7}, we have that $\lambda \sim \phi \oplus \psi$. Indeed, let $T= \left[
             \begin{matrix}
              1 & 1 & 1\\
              1 & -2 & 1 \\
              1 & -1 & -2
              \end{matrix}
                   \right]. $
It can be checked that $(\phi \oplus \psi)_{x} T= T\lambda_{x}$ for all $x \in \mathbb{P}_{3}$.\\
Since the finite enveloping group $\overline{G_{\mathbb{P}_{3}}} \cong \mathbb{S}_{3}$ has trivial center, then by Theorem \ref{Theorem 2.8} the representations $\phi$ and $\psi$ are strong.\\
It is well known that the group $\mathbb{S}_{3}$ has three (up to equivalence) irreducible representations, then  from Corollary \ref{Collorary 2.3}, the number of irreducible strong representations of $\mathbb{P}_{3}$ is equal to $3$. Previously, we found two irreducible strong representations of $\mathbb{P}_{3}$, from the knowledge of representations of finite groups we can find the last one. The other irreducible representation of $\mathbb{S}_{3}$ is $\bar{\tau}: \mathbb{S}_{3} \longrightarrow GL(\c^{*}) $ defined by
\begin{center}
    $\bar{\tau}_{\sigma} := \left\{ 
    \begin{array}{lcc}
             1 &   if  & \sigma  \ is \ even \\
              -1 &  if & \sigma  \ is \ odd \\
               \end{array} \right.$
\end{center}
Thus, we have a rack representation $\tau: \mathbb{P}_{3} \longrightarrow Conj(GL(\c^{*}))$ defined by
\begin{center}
    $\tau_{(2 \ 3)}:= \bar{\tau}_{g_{(23)} \langle g_{(12)}^{2}\rangle} = \bar{\tau}_{(2 \ 3)} = -1$, \\ 
    $\tau_{(1 \ 3)}:= \bar{\tau}_{g_{(13)}\langle g_{(12)}^{2}\rangle} = \bar{\tau}_{(1 \ 3)} = -1$,\\
     $\tau_{(1 \ 2)}:= \bar{\tau}_{G_{(12)}\langle g_{(12)}^{2}\rangle} = \bar{\tau}_{(1 \ 2)} = -1$.
\end{center}
 Therefore, we have $\tau_{(i \ j)} = -1$ for all $(i \ j) \in \mathbb{P}_{3}$. From Theorem \ref{Theorem 2.8}, the representation $\tau$ is strong. Hence, the permutation quandle $\mathbb{P}_{3}$ has three (up to equivalence) irreducible strong representations. 

\end{eg}

\begin{eg}
      Let $X=\{1,2,3\}$ be the rack given in  Example \ref{Example 2.7}. The operation of this rack is defined as $i \vartriangleright j = \sigma(j)$, for all $i,j \in X$, where $\sigma = (1 \ 2\ 3) \in \mathbb{S}_{3}$.The regular representation $\lambda: X \longrightarrow Conj(GL(\c X))$ is defined by

\begin{align*}
       \lambda_{1} = \lambda_{2}=\lambda_{3}=  \left[
             \begin{matrix}
              0 & 0 & 1 \\
              1 & 0 & 0\\
              0 & 1 & 0
             \end{matrix}
                   \right].
\end{align*}
Previously, we found the finite enveloping group of this rack, which is $\overline{G_{X}} = span \{ g_{1}\langle g_{1}^{3} \rangle\} \cong \z_{3}$. Therefore, $\lambda$ induces a group representation $\bar{\lambda}: \overline{G_{X}} \longrightarrow GL(3,\c)$ defined by:

\begin{align*}
       \bar{\lambda}_{g_{1}\langle g_{1}^{3} \rangle} = \bar{\lambda}_{g_{2}\langle g_{1}^{3} \rangle}= \bar{\lambda}_{g_{3}\langle g_{1}^{3} \rangle}=  \left[
             \begin{matrix}
              0 & 0 & 1 \\
              1 & 0 & 0\\
              0 & 1 & 0
             \end{matrix}
                   \right].
\end{align*}

The irreducible representations of $\overline{G_{X}} \cong \z_{3}$ are one dimensional and they are the cube roots of unity, that is $\bar{\rho}( g_{1}\langle g_{1}^{3} \rangle) = 1$, $\bar{\psi}( g_{1}\langle g_{1}^{3} \rangle) = w$, $\bar{\phi}( g_{1}\langle g_{1}^{3} \rangle) = w^{2}$ where $w= e^{2\pi i/3}$. The character table of this group representation is 

\begin{center}
\begin{tabular}{c|c|c|c}
 & ${\langle g_{1}^{3} \rangle}$ & ${g_{1}\langle g_{1}^{3} \rangle}$  & ${g_{1}^{2}\langle g_{1}^{3} \rangle}$\\ \hline
${\chi_{\bar{\rho}}}$ & 1 & 1 & 1 \\ \hline
${\chi_{\bar{\psi}}}$ & 1 & w & $w^{2}$ \\ \hline
${\chi_{\bar{\phi}}}$  & 1 & $w^{2}$ & $w$ \\
\hline
${\chi_{\bar{\lambda}}}$  & 3 & 0 & 0 

\end{tabular}
\end{center}

Where we have that $\bar{\lambda} \sim \bar{\rho} \oplus \bar{\psi} \oplus \bar{\phi}$. Now, the irreducible representations of $\overline{G_{X}}$ induces irreducible representations of the rack $X$ defined by 
\begin{center}
    $\rho_{i}= 1 \ \ \forall i \in X$,\\ $\psi_{i}= w \ \ \forall i \in X$, \\
    $\phi_{i}= w^{2} \ \ \forall i \in X$.
\end{center}
Therefore, we have that $\lambda \sim \rho \oplus \psi \oplus \phi$. The representation $\rho$ is the trivial one, so it is strong. We claim that the representations $\psi$ and $\phi$ are also strong. Indeed, let $\{i_{1},\dots,i_{k}\}$ be a stabilizing family of the rack $X$. Note that for every $j \in X$ we have that
\begin{align*}
   j &= i_{k} \vartriangleright ( \cdots \vartriangleright (i_{2} \vartriangleright (i_{1} \vartriangleright j)) \cdots) = i_{k} \vartriangleright ( \cdots \vartriangleright (i_{2} \vartriangleright \sigma(j) \cdots) \\
    &= i_{k} \vartriangleright ( \cdots (i_{3} \vartriangleright \sigma^{2}(j) \cdots) =\sigma^{k}(j)
\end{align*}
Therefore, $\sigma^{k}=id$. Since the order of $\sigma$ is 3, then $k$ is of the form $k=3m$, for some $m \in \n$. Then,
\begin{align*}
    \psi_{i_{k}} \cdots \psi_{i_{1}} = w^{k} = w^{3m} = 1\\
    \phi_{i_{k}} \cdots \phi_{i_{1}} = w^{2k} = w^{6m} = 1
\end{align*}
Thus, the representations are strong. Hence, the rack $X$ has 3 (up to equivalence) irreducible strong representations.

\end{eg}

Elhamdadi and Moutou, in \cite{Elhamdadi}, stated the theorem:

\textit{``Theorem 9.11: Every strong irreducible representation of a finite connected involutive rack is one–dimensional"}.  

The reader can see that Example \ref{Example 2.10} is a counterexample to this theorem. Indeed, the permutation quandle $\mathbb{P}_{3}$ is finite, connected and involutive, however it has one (up to equivalence) strong irreducible representation of dimension two. Furthermore, since  for every $k \in \n$  we have $\mathbb{S}_{k} \cong \overline{G_{\mathbb{P}_{k}}}$; then for all $k \in \n$, the permutation quandle $\mathbb{P}_{k}$, which is finite, involutive and connected, has at least one irreducible strong representation with dimension larger than one. Therefore, we can form an infinite family of strong representations that contradicts such theorem.

\end{document}